 \documentclass[draft]{article}

\usepackage{amsmath,amsfonts,amsthm,amssymb,amscd,color}
\renewcommand{\Re}{\mathop{\rm Re}\nolimits}
\renewcommand{\Im}{\mathop{\rm Im}\nolimits}

\newcommand{\beq}{\begin{equation}}
\newcommand{\ee}{\end{equation}}

\theoremstyle{plain} \newtheorem{theorem}{Theorem}[section]
\newtheorem{lemma}[theorem]{Lemma}
\newtheorem{proposition}[theorem]{Proposition}
 \theoremstyle{definition}
\newtheorem{definition}[theorem]{Definition} \theoremstyle{remark}
\newtheorem{remark}[theorem]{Remark}

\newcommand{\R}{{\mathbb R}} \newcommand{\U}{{\mathcal U}}

\newcommand{\Z}{{\mathbb Z}}

\newcommand{\Ph}{{\mathcal P}}

\newcommand{\resto}{{\mathcal R}}
\def\im{{\rm i}}
\newcommand{\N}{{\mathcal N}}

\newcommand{\C}{\mathbb{C}}

\def\uno{{\kern+.3em {\rm 1} \kern -.22em {\rm l}}}

\def\({\left(}
\def\){\right)}
\def\<{\left\langle}
\def\>{\right\rangle}

\numberwithin{equation}{section}

\begin{document}

\title{On orbital instability of spectrally stable vortices of the NLS in the plane}

\author {Scipio Cuccagna and Masaya Maeda}

\date{\today}
\maketitle
\begin{abstract} We   explain    how  spectrally stable  vortices of the   Nonlinear Schr\"odinger Equation  in the plane  can be  orbitally unstable.
This relates to the  nonlinear Fermi
golden rule, a mechanism which exploits the nonlinear interaction between discrete and continuous
modes  of the NLS.

\end{abstract}

\section{Introduction}
\label{section:introduction}

In this paper we consider the nonlinear Schr\"odinger equation (NLS):
\begin{equation}\label{NLS}
 \im u_{t }=-\Delta u +Vu +\beta  (|u|^2) u , \, u(0,x)=u_0(x), \, (t,x)\in\mathbb{ R}\times
 \mathbb{ R}^d,
\end{equation}
 with $  V$ a  real valued Schwartz function.
We are interested in bound states, which are solutions of \eqref{NLS} of the form
$  u(t,x) =    e ^{ \im \omega t  }  \phi _{\omega } (x)$ with $\omega\in\R$.
When  $\phi_\omega$ is real valued and of fixed sign then we call $e^{\im \omega t}\phi_\omega$   a ground state. In all other cases   we call it an excited state.
In the $d=2$ case  on which we focus in this paper, and  with  $V(x)=V(|x|)$,  an important class of excited states, which we call  vortices, involves    solutions of the form
\begin{equation}\label{vortex}
  u(t,x) =    e ^{ \im  \omega t }  \phi _{\omega } (x) \text{ with  }  \phi _{\omega } (x) =e ^{\im m \arg (x)   } \psi _{\omega } (|x|) \text{ with  }  \psi _{\omega }:\R^2 \to \R ,
\end{equation}
with $ \phi _{\omega } (x)$  smooth and rapidly decreasing to 0 at infinity and   $m\geq 1$.
In the sequel, we will always assume that there is a family of bound states $\phi_\omega$ for $\omega$ in some open interval $\mathcal O\subseteq \R_+=(0,\infty)$ (see (H4) in section \ref{sect:statement}).
We will study the following  classical notion of   stability, \cite{GSS1, W2}.

\begin{definition}[Orbital stability]\label{def:orbst}
A bound state $ e ^{ \im \omega t  } \phi _\omega $ of
\eqref{NLS}  is {\it orbitally stable} if
\begin{equation*}\text{$\forall$ $\epsilon>0$,  $\exists$ $\delta>0$ s.t.    $\|\phi _\omega -u_0\|_{H^1}<\delta$}\Rightarrow
\sup_{t>0}\inf_{s\in\R }\|e^{\im s }\phi _\omega -u(t)\|_{H^1}<\epsilon
\end{equation*}
where $u$  is the solution of \eqref{NLS} with $u(0)=u_0$.

\end{definition}
The orbital stability  of bound states has been
extensively    studied, mainly using    two   tools:
       Lyapunov functions;    linearized operators.

It is well known that \eqref{NLS} conserves   the energy
\begin{equation} \label{eq:energyfunctional}\begin{aligned}&
 E(u):=\frac 1 2 \int_{\R^2} |\nabla u|^2 + V|u|^2 + B(|u|^2)\,dx  \end{aligned}
\end{equation}
where $B(0)=0$ and $B'(s)=\beta(s)$,  and the mass  \begin{equation}\label{eq:charge}Q(u):=2 ^{-1}\int _{\R ^2}
  |  u | ^2dx     \text{ (we will set $q(\omega ) =Q(\phi _\omega ) $  and $q'(\omega ) = \frac{d}{d\omega}q(\omega )$)}.
\end{equation}
Using these conservations laws, if $\phi_\omega$ is a strict local minimizer  up to constant phase $e^{\im \theta}$ of $E$ under the constraint $\|u\|_{L^2}=\|\phi_\omega\|_{L^2}$, then it has been shown that  $e^{\im \omega t}\phi_\omega$ is orbitally stable, see
 \cite{CL,GSS1,GSS2,W1}. We are interested on standing waves not covered by this classical result. We will use the following notion.

\begin{definition}\label{def:nontrap}
A    bound state  $e^{\im \omega t}\phi_\omega$ is \textit{not trapped by the energy}
if for any $\varepsilon>0$   there exists an $u_\varepsilon$ satisfying $\|\phi_\omega-u_\varepsilon\|_{H^1}<\varepsilon$, $\|u_\varepsilon\|_{L^2}=\|\phi_\omega\|_{L^2}$ and $E(u_\varepsilon)<E(\phi_\omega)$.
\end{definition}
The results  \cite{CL,GSS1,GSS2,W1} do  not cover   the case when  $\phi_\omega$ is a local but not strict  minimizer. They also leave  unsolved
the   case when $\phi_\omega$ is not trapped by the energy, which we will discuss here.

In order to study the stability, it is natural to consider the linearized operator $\mathcal L_\omega$ of $\phi_\omega$ (see \eqref{eq:linearizationL} for the explicit form).
Indeed, if $\mathcal L_\omega$ has unstable modes (spectrum with positive real part), then $e^{\im \omega t}\phi_\omega$ is orbitally unstable.
(Even though this   may look trivial, it is not, and it was proved rather recently
 first in 2D by \cite{M1} and later in general
 by \cite{georgiev}).  Classical   papers proving   {orbital} instability
of  solitary waves   by first proving   their  {spectral } instability are \cite{Grillakis,GSS2,jones,jones1}.

In the case of  ground states, except for the  degenerate cases   when $ q'(\omega) =0$, only the above two cases (trapped by the energy or linearly unstable) occur. That is, if $ q'(\omega) <0$ then $\mathcal L_\omega$ has   an unstable mode    while if $ q'(\omega) >0$ it is
trapped by the energy, \cite{GSS1,GSS2}.
For the degenerate  case $q'(\omega)=0$, see \cite{CP03, Maeda12JFA, Ohta11JFA}.

Excited state  are usually   not trapped by the energy and furthermore there are cases when $\mathcal L_\omega$ has no unstable modes.
For example, if $-\Delta+V$ has $\mathbf{n}$ simple negative eigenvalues $\{e_1<e_2<\cdots<e _{\mathbf{n}}(<0)\}$  then, if $2e_j<e_1$ for some $j\geq 2$, the excited states bifurcating form $e_j$ are not trapped by the energy and are spectrally stable. Even if    spectrally stable they are
orbitally unstable,   \cite{CM3}.   Prior to \cite{CM3}
no systematic proof of this orbital  instability was available.  The series
  \cite{GW1,GW2,GP, NPT,SW4,T,TY1,TY2,TY3,TY4}, which stemmed from  \cite{BP2,SW3},
 was able to treat  only case  $2e_2>e_1$ where, generically, excited states are  spectrally unstable, see \cite{CPV,GP,NPT,TY4}.

Vortices  \eqref{vortex}  of \eqref{NLS}  in the    important  pure  power case $\beta  (|u|^2) u=-|u| ^{p-1}u$    have been considered in
    \cite{M2,M3,M4}   which have various instability results, always by   proving first {spectral } instability.
Another important example is given by \cite{pego,Quiroga,Towers}   for vortices for the cubic-quintic nonlinearity  \begin{equation}\label{NLS3-5}
 \im u_{t }=-\Delta u   - (|u|^2 -|u|^4) u , \, u(0,x)=u_0(x), \, (t,x)\in\mathbb{ R}\times
 \mathbb{ R}^2,
\end{equation}
(for a review paper see also \cite{CMM}, for work on spinning solitons in 3D see \cite{mihalache}; see also \cite{Bellazzini}).
\cite{pego,Towers}  show  numerically
   for some values   $|m|\ge 1$ the  existence of  a critical value $ \omega _{cr}$ such that for $\omega < \omega _{cr}$ the vortices   are spectrally unstable and for $\omega \ge  \omega _{cr}$ are spectrally
stable.  In the simulations in \cite{pego}  the spectrally stable  vortices for $\omega >  \omega _{cr}$  appear stable while  in \cite{Towers}   for $m=3$
appear  to  slowly develop  instabilities.     This latter observation appears  consistent  with the more recent numerical observations in \cite{CKP,KPS},
in turn based on the instability theory in \cite{Cu3}.  The theory  in \cite{Cu3}   is centered on the notion of Krein signature, which we introduce in Lemma \ref{lem:basis}  (although the standard definition is in the proof of Lemma \ref{lem:bif}).
In this paper, we generalize \cite{Cu3} by using a simple idea from \cite{CM3}.
We will need the following notion.
\begin{definition}[Conditional asymptotic stability]\label{def:conas}
We say that a bound state $e^{\im \omega t}\phi_\omega$ is {\it conditionally asymptotically stable} if  there
exist constants $\epsilon _0>0$  and $C_0>0$  s.t.\ if  $u\in C^0([0,\infty ), H^1)$ is a solution of \eqref{NLS}
 with $  \sup _{t\ge 0}\inf _{\vartheta \in \R }  \| u (t) -e^{\im \vartheta}\phi _{\omega  }\| _{H^1}< \epsilon <\epsilon _0$   then  there exist  $\omega
_+\in\mathcal{O}$, $\theta\in C^1(\R ;\R)$ and $h _+ \in H^1$
with $\| h_+\| _{H^1}+|\omega _+ -\omega |\le C_0 \epsilon  $
such that
\begin{equation}\label{scattering}
\lim_{t\to   \infty}\|u(t )-e^{\im \theta(t)}\phi_{\omega
_+}-e^{\im t\Delta }h _+\|_{H^1}=0 .
\end{equation}
\end{definition}

Then we have the following  orbital instability result.
\begin{theorem}\label{theorem-1.1}
 Consider a bound state $e^{\im \omega t }\phi _{\omega  }$ and suppose that Hypotheses (H1)--(H5)
 in Section \ref{sect:statement} below are satisfied.    Then
if $e^{\im \omega t }\phi_{\omega  }$ is not trapped by the energy and  is conditionally asymptotically stable it  is also orbitally unstable.
\end{theorem}
\proof Here the key hypothesis is  (H5), i.e.  $q'(\omega )\neq 0$ and hence either $q'(\omega )> 0$  for all $\omega$ or $q'(\omega )< 0$  for all $\omega$.
  We prove the theorem by contradiction.

  \noindent Assume that the statement is false and that   there  is a $\omega _1
  \in \mathcal{O}$ s.t.\ the standing wave  $e^{\im \omega_{1} t }\phi _{\omega_{1} }$  is orbitally stable. Since  $e^{\im \omega_1 t }\phi_{\omega _1 }$ is not trapped by the energy
there are initial data   $u_0$ arbitrarily close to $\phi _{\omega_{1} }$   such that $E(u_0)< E(\phi _{\omega_{1} })$  and $Q(u_0)= Q(\phi _{\omega_{1} })$. We can apply the Conditional asymptotic stability, Definition \ref{def:conas}, and conclude by the conservation of the mass that
\begin{align*}  q(\omega _1)=Q(u (t))  =\lim _{t \to \infty } Q(u (t))= q(\omega _+ ) + 2^{-1} \| h _{+}\|_{L^2}^2 \ge  q(\omega _+).
\end{align*}
Similarly, by the conservation of the energy  we have
\begin{align}\label{enin}
E(\phi _{\omega_{1} })>E (u_0) =\lim _{t \to \infty } E(u (t))=
E(\phi _{\omega_{+} }) +2^{-1}\|\nabla h _{+}\|_{L^2}^2\geq E(\phi _{\omega_{+} }).
\end{align}
By $
 \nabla  {E} (\phi _\omega   )= -\omega  \nabla  Q (\phi _\omega   )   $
 we have $\frac{d}{d\omega} {E} (\phi _\omega   ) =-\omega  q' ( \omega   )$  (recall the notation $q(\omega):=Q (\phi _\omega   )   $ in \eqref{eq:charge}). On the other hand, by  (H5) we know that $q' ( \omega   )$
 has a fixed sign. So, since $\omega >0$,  both $\omega \to {E} (\phi _\omega   )$ and  $\omega \to q(\omega   )$ are strictly monotonic,   one   increasing
  and  the other   decreasing. This means that $q(\omega _1)   \ge  q(\omega _+)$  implies that $   E(\phi _{\omega_{1} })\le E(\phi _{\omega_{+} })$. But this contradicts      $  E(\phi _{\omega_{1} })> E(\phi _{\omega_{+} })$ in \eqref{enin}. This contradiction shows that   $e^{\im \omega_{1} t }\phi _{\omega_{1} }$  is not  orbitally stable  for any $\omega _1
  \in \mathcal{O}$.\qed

\begin{remark}
Notice that (after modifying the assumption (H4)) Theorem \ref{theorem-1.1} holds for arbitrary dimension and we do not require that the nonlinear bound states are vortices.
\end{remark}

\begin{remark}
We do not claim that conditional  asymptotic stability is necessary for the instability of excited states.
\end{remark}

\begin{remark}
For discrete NLS  there are examples of bound states which are not trapped by the energy, are orbitally stable, but are not asymptotically stable, see \cite{maeda, maedamasaki}.
\end{remark}


We now start the discussion on how to apply {Theorem} \ref{theorem-1.1}.
We first discuss a sufficient condition for the standing waves to be not trapped by the energy.
One natural way is to look at the Taylor expansion of the energy, which is often used in the study of orbital stability.
Set $S_\omega(u)=E(u)+\omega Q(u)$ (notice that if $Q(u)=Q(v)$, $E(u)>E(v)$ is equivalent to $S_\omega(u)>S_\omega(v)$).
Then, since $e^{\im \omega t}\phi_\omega$ is a bound state if and only if $\nabla S_\omega (\phi_\omega)=0$, we have $S_\omega(\phi_\omega+v)=S_\omega(\phi_\omega)+\frac 1 2 \<\nabla ^2 S_\omega(\phi_\omega)v,v\>+o(v^2)$.
Roughly speaking, we have one constraint $Q(\phi_\omega+v)=Q(\phi_\omega)$ which may eliminate at most one negative direction of $\nabla^2 S_\omega(\phi_\omega)$.
Therefore, if $\nabla^2 S_\omega (\phi_\omega)$ has more than two negative eigenvalues, $\phi_\omega$ is not trapped by the energy.
Further, for the case $\nabla^2 S_\omega (\phi_\omega)$ has one negative eigenvalue, the trapping/nontrapping can be determined by the sign of $q'(\omega)$.
That is, if $q'(\omega)>(<)0$ then $e^{\im \omega t}\phi_\omega$ is trapped (not trapped) by the energy.

Another viewpoint  which we adopt in this paper is to study the linearized operator $\mathcal L_\omega =J\nabla^2 S_\omega (\phi_\omega)$ given in \eqref{eq:linearizationL}, where $J=\begin{pmatrix}0 & 1\\ -1 & 0 \end{pmatrix}$.
Here, we have extended $\nabla^2 S_\omega (\phi_\omega)$ to be a matrix to make it self-adjont, see section \ref{sect:statement}.
Since $\mathcal L_\omega$ has the form $J(-\Delta+\omega)+$``rapidly decaying potential", we see that $\im \sigma_{\mathrm{e}}(\mathcal L_\omega)=(-\infty,-\omega]\cup[\omega,\infty)$.
Here $\sigma_{\mathrm{e}}(\mathcal L_\omega)$ is the   essential spectrum of $\mathcal L_\omega$.
Since we are only interested in the case that $\mathcal L_\omega$ is spectrally stable we assume $\sigma(\mathcal L_\omega)\subset \im \R$.

\noindent In general  no eigenvalues are expected in the interior of the essential spectrum.
References \cite{AS11JMP, MS11N} have computational proofs about the absence of
embedded eigenvalues when $\phi _\omega$ is a ground state for some equations.\cite{pego}  does not discuss explicitly
the issue of embedded eigenvalues   but they seem to be excluded, see  {Remark}
\ref{rem:pego8-12}  below for some further comments.
\cite{CPV} proves   that  in some \textit{generic} sense  embedded eigenvalues
 do not exist because they are unstable, see also \cite{Grillakis,TY4}. The proof     is similar
 to the   discussion of the well established instability of embedded eigenvalues in the case of self-adjoint operators, cfr. \cite{howland}. However it is not clear whether    taking  for example $V$ generic in the cubic--quintic NLS would make $\mathcal{L}_{\omega}$ generic in the sense of \cite{CPV}.
 We think that embedded eigenvalues (with some negative Krein signature) can exist, but that
 they are unstable (we conjecture the non existence of embedded eigenvalues of positive Krein signature).
 Hence the assumption of  absence of embedded eigenvalues seems reasonable.

 \noindent Similarly,     the edge of the   essential spectrum is   shown in  \cite{CP05JMP,Vougalter10MMNP} generically   to be neither an  eigenvalue  nor resonance.  Again, we did not check
 whether taking $V$ generic makes the $\mathcal{L}_{\omega}$ associated to the cubic--quintic NLS   generic in that sense, although this is even more likely   than for the issue of embedded eigenvalues.  In \cite{pego} there is no discussion about the edge of the   essential spectrum where the presence of a resonance or
 an eigenvalue would affect the computations discussed in pp. 371--372.  We think that
 assuming absence of  eigenvalue   or resonance at the edge is reasonable.

\noindent  We will further assume that the kernel is spanned by $J\phi_\omega$ and all nonzero eigenvalues  (which are assumed to be finitely many but in fact this can be proved to be the case,  see the comments under (H11)  in Sect. \ref{sect:statement})   have the same algebraic and geometric dimensions which are finite (in the presence of a nonzero  imaginary eigenvalue
whose two dimensions differ we would have a sort of linear instability, see Remark \ref{rem:pego8-12} for further comments).
Now, let $\xi$ be the eigenfunction of $\mathcal L_\omega$ associated to $\im\lambda$ with $\lambda>0$
(for eigenfunctions of $\im \lambda$ with $\lambda<0$, we have the symmetry.
Thus it suffices to consider the positive case).
Then, the ``energy" of $\xi$ is $\<\nabla^2 S_\omega(\phi_\omega)\xi,\bar \xi\>=\Omega(\mathcal L_\omega \xi,\bar \xi)=\im \lambda \Omega(\xi,\bar \xi)$, where $\Omega=\<J^{-1},\cdot,\cdot\>$ (see \eqref{eq:Omega}).
It is known that we can normalize $\xi$ s.t.\ $\Omega(\xi,\bar \xi)=\im s$ for $s\in \{1,-1\}$.
Therefore, if $s=1$, $\xi$ have a negative energy and if $s=-1$, $\xi$ have a positive energy.
Now, set $n(\nabla^2 S_\omega(\phi_\omega))$ be the number of negative eigenvalues of $\nabla ^2 S_\omega(\phi_\omega)$.
It is known that $n(\nabla^2 S_\omega(\phi_\omega))$ can be represented as
\begin{align*}
n(\nabla ^2 S_\omega (\phi_\omega))=p(q')+2N_r^-+N_i+2N_c,
\end{align*}
where $p(q')=1$ if $q'>0$ and $0$ if $q'<0$, $N_r,N_i,N_c$ are the number of real, imaginary and complex eigenvalues of $\mathcal L_\omega$ (See \cite{CPV, Vougalter10CMB}).
In our situation, $N_i=N_c=0$.
Therefore, if $N_r^-=0$, $\phi_\omega$ will be trapped by the energy and if $N_r^-\geq 1$, the $\phi_\omega$ is not trapped by the energy.
Thus, by the above discussion, the trapping/nontrapping of the energy is determined by the existence of eigenvalue of $\mathcal L_\omega$ with negative energy.
Here, we will give a direct proof of the nontrapping when there exists a negative energy eigenvalue.

\begin{proposition}\label{prop:h14}
Assume hypothesis (H1)--(H11)     and (H14)   in Section \ref{sect:statement}.
Then the standing waves are not trapped by the energy.
\end{proposition}

\begin{proof}  Here (H14) is the key hypothesis which is roughly  stating that   for any $\omega$  there is
one eigenvalue $\im \lambda _j \in \R_+$ with negative Krein signature (i.e.\ the corresponding eigenfunction $\xi_j$ satisfies $\Omega(\xi_j,\bar \xi_j)=1$).

Fix $\omega_1\in \mathcal O$ and choose $\alpha(\epsilon)$ to satisfy $Q\((1-\alpha(\epsilon))\phi_{\omega_1}+\epsilon(\xi_j(\omega_1)+\bar \xi_j(\omega_1))\)=Q(\phi_{\omega_1})$.
Since $\xi_j(\omega),\bar \xi_j(\omega)\in N_g(\mathcal L_{\omega_1}^*)^\perp$ (see \eqref{eq:Lstar}), we see that $$Q\((1-\alpha(\epsilon))\phi_{\omega_1}+\epsilon(\xi_j(\omega_1)+\bar \xi_j(\omega_1))\)=(1-\alpha(\epsilon))^2Q\(\phi_{\omega_1}\)+\epsilon^2Q\((\xi_j(\omega_1)+\bar \xi_j(\omega_1))\).$$
Thus, we can conclude $\alpha(\epsilon)\sim \epsilon^2$.
Consequently, we have
\begin{align*}
&E\((1-\alpha(\epsilon))\phi_{\omega_1}+\epsilon(\xi_j(\omega_1)+\bar \xi_j(\omega_1))\)-E(\phi_{\omega_1})\\
&=S_{\omega_1}\((1-\alpha(\epsilon))\phi_{\omega_1}+\epsilon(\xi_j(\omega_1)+\bar \xi_j(\omega_1))\)-S_{\omega_1}(\phi_{\omega_1})\\&=
\frac {\epsilon^2}{2}\<\nabla^2 S_{\omega_1}(\phi_{\omega_1}) \(\xi_j(\omega_1)+\bar \xi_j(\omega_1)\),\xi_j(\omega_1)+\bar \xi_j(\omega_1)\>+o(\varepsilon^2).
\end{align*}
Finally, since $\mathcal L_{\omega_1}=J \nabla^2 S_{\omega_1}(\phi_{\omega_1})$, we have
\begin{align*}
&\<\nabla^2 S_{\omega_1}(\phi_{\omega_1}) \(\xi_j(\omega_1)+\bar \xi_j(\omega_1)\),\xi_j(\omega_1)+\bar \xi_j(\omega_1)\>= \Omega\(\mathcal L_{\omega_1} \(\xi_j(\omega_1)+\bar \xi_j(\omega_1)\),\xi_j(\omega_1)+\bar \xi_j(\omega_1)\)\\&=
\im \lambda_j \Omega(\xi_j(\omega_1),\bar \xi_j(\omega_1))-\im \lambda_j  \Omega(\bar \xi_j(\omega_1),\xi_j(\omega_1))=-2 \lambda_j<0.
\end{align*}
Therefore, we have the conclusion.
\end{proof}

We now turn our attention to the conditional asymptotic stability.
Its proof is very close to the proof of asymptotic stability of ground states in  \cite{Cu0,Cu2} and we will need assumptions for the spectrum of $\mathcal L_\omega$ discussed above (i.e. absence of embedded eigenvalues, $\pm \im \omega$ are neither eigenvalues nor resonance, the generalized kernel is 2 dimensional, all nonzero eigenvalues have nonzero energy, see section \ref{sect:statement}).
We will consider only    the 2D case.
\begin{theorem}\label{th:asstab}
Assume hypothesis (H1)-(H13) in Section \ref{sect:statement}.
Then, we have the conditional asymptotic stability.
\end{theorem}

The main example for the theory developed in this paper that we have in mind   is the cubic quintic NLS  \eqref{NLS3-5}
 and specifically   perturbations obtained adding a radial potential $V(|x|)u$.
  As we mentioned above,  in \cite{pego} it is shown that for any $m=1,2,3,4,5$
  there is a critical value   $ \omega _{cr}$ such that for $\omega < \omega _{cr}$ the vortices   are spectrally unstable and for $\omega \ge  \omega _{cr}$ are spectrally
stable.   \cite{pego} proves numerically that the linearization   $\mathcal{L}_{\omega}$ has an eigenvalue of negative Krein signature, i.e. our hypothesis (H14) is satisfied.  Indeed in \cite{pego}  it is shown  numerically  that for   \eqref{NLS3-5} the spectral instability for $\omega <  \omega _{cr}$
occurs because a pair of eigenvalues on $\im \R$    coalesce     as $\omega\to  \omega_{cr}  ^+$ and then for $\omega < \omega_{cr}$ form   two eigenvalues which  exit  $\im \R$  in  opposite directions. This can happen only if the eigenvalues  for  $\omega >  \omega _{cr}$ do not have the
same Krein signature.
Thus for $\omega $ close to and larger than $  \omega _{cr}$  at least one imaginary eigenvalue has
negative signature. See below in {Lemma} \ref{lem:bif} for a more precise discussion.
Our results do not apply directly to the 2D  cubic quintic NLS \eqref{NLS3-5} because of its translation invariance. However, as we show in Section \ref{sec:applications}, when we add to \eqref{NLS3-5}
a small linear potential, then   we obtain an equation which satisfies hypothesis (H14) and which is not translation invariant.
When we take the small potential with a nondegenerate minimum at the origin, then for $\omega >\omega  _{cr}$   we   obtain   spectrally stable vortices.
 Our hypotheses  (H1)--(H7) are either obvious or we know they are true
 as a consequence of the numerical experiments  in \cite{pego}. The hypotheses (H9)--(H13),
 while probably generically true,   ought to be checked numerically. We will say more later
 about them, especially  (H13), the most delicate and least analyzed.
The conclusion that
 we can draw is that, assuming that indeed  (H9)--(H13) are   true, then
for   $\omega$ close to and larger than $\omega  _{cr}$ the vortices of appropriate perturbations of the cubic quintic NLS \eqref{NLS3-5} are  not trapped by the energy and are conditionally asymptotically stable. And hence  by {Theorem} \ref{theorem-1.1} they are orbitally unstable.

We have seen that Theorem   \ref{theorem-1.1}  is a simple consequence
of non trapping by energy (by Proposition  \ref{prop:h14} a consequence of
the existence of one eigenvalue of negative Krein signature)
and of
the Conditional asymptotic stability. The latter follows from   (H1)--(H13)  by Theorem  \ref{th:asstab}.  A more precise formulation
of Theorem  \ref{th:asstab} is in {Proposition} \ref{prop:asstab}.  The Conditional asymptotic stability,  like the asymptotic stability of ground states proved in \cite{Cu2},
is due to a mechanism of loss of energy of  discrete modes  related to the Fermi golden rule (FGR) and to linear scattering of
  the continuous modes.

   The    FGR  was first introduced in   \cite{BP2,sigal,SW3} and we will discuss it shortly.
But first we discuss   linear scattering, which  is based on a number of results on the group
$e^{t\mathcal{L}_{\omega}}$ associated to the linearization of the NLS at the vortex.
The results needed here are quoted (here we focus on 2D) from \cite{CT} and require that $\mathcal{L}_{\omega}$  should  not have eigenvalues (and resonances)
  in the essential spectrum, see hypotheses (H9)--(H10) below.
  The results in  \cite{CT}
    allow to say that, for all practical purposes, the restriction of  $e^{t\mathcal{L}_{\omega}}$  on the continuous spectrum part, behaves like $e^{\im t(-\Delta +V)}$ restricted to its continuous spectrum part.
  (H9)--(H10) are
  probably generically true, but nonetheless ought  to be  either proved or checked
  numerically on any given example. The fact that we cannot treat translation  (the asymptotic stability  result on moving solitons in \cite{Cu4} has not been proved in dimensions
  1 and 2)   depends on the specific way in which scattering of continuous modes is   proved. Probably  there is a  simpler and more robust way to prove dispersion  using
  virial inequalities and the theory of Martel and Merle \cite{MaM0,MaM1}. This should require  fewer hypotheses and lead to similar results.  This approach
    has been very successful   in the context of KdV equations where it
  has improved the result by Pego and Weinstein \cite{PW}.   The  theory is proving successful also in different contexts, see \cite{BGS,cote,GS,Ma,KMM}.
  However here we follow our standard approach and we use material
  in   \cite{CT}.

We now turn to the FGR.
It can be easily seen under
appropriate coordinate systems.
It leads to   nonlinear interactions
      between discrete an continuous modes of the NLS which are responsible for
    energy leaking out from  the  discrete modes.

 When we discuss the FGR  we need to separate two distinct issues,
     as we will see with a simple example below.
  One  issue is the fact that certain coefficients of the system have a 2nd power structure. This
  has been proved in \cite{Cu2}. See also \cite{Cu0,CM3} for generalizations and references.
 A separate issue,   is whether   or not these
2nd powers, which are non negative,  are also  strictly positive. There might be cases when this is not true, but in general
we expect that they are strictly positive. We do not have the expertise to run numerical
tests, but a simple model might clarify this point (for other examples of FGR see also
 the survey \cite{W3}).

 For $(z,h)\in \C\times H^1(\R^2, \C)$, consider the Hamiltonian
\begin{equation}   \label{model1}    \begin{aligned} &
  \mathcal{H} (z,h )=
	   |z |^2+  \|    \nabla h \| ^{2}_{L^2}  +   |z|^2   \overline{{z}}    \int _{\R^2}
  G   (x) h  (x) dx  +   |z|^2   z
  \int _{\R^2}\overline{G}   (x)\overline{h } (x) dx,
   \end{aligned}
\end{equation}
where $G$ is a $\C$-valued Schwartz function and  the symplectic form $\im dz\wedge d\bar z + 2 \< \im dh, d  h\>$
where \begin{equation} \label{eq:hermitian}\begin{aligned}& \langle f,g\rangle =\Re  \int
_{\mathbb{R}^2}f(x)  \overline{g}(x)  dx =  \int
_{\mathbb{R}^2}(   f_1(x)    {g}_1(x) +   f_2(x)    {g}_2(x) )   dx   \text{ for $f,g:\mathbb{R}^2\to \C(=\mathbb{R}^2)$ }
 .  \end{aligned}\end{equation}
Then  we have  the Hamiltonian system
   \begin{align}  \label{intstr0} &
  \im \dot h = - \Delta   h + |z|^2   z  \overline{G},
  \\&
\im \dot z
=  z   +   2 |z|^2
\int _{\R^2} h (x)
  {G}(x)  dx   +             z^2
\int _{\R^2} \overline{h}(x)
  \overline{{G}}(x)  dx   . \label{intstr11}
\end{align}
The solution of the linearized equation around $(0,0)$ is $(z,h)=(e^{-\im t}z(0),e^{\im t \Delta}h(0))$.
Therefore, at the linear level, we do not see the asymptotic stability of $(0,0)$.
In the following, we sketch  heuristically  why the equilibrium $(0,0)$ is asymptotically stable thanks to
the FGR and scattering of the continuous mode.
First, if we assume $z(t)=e^{-\im t}z(0)$, then $-|z|^2z R^+_{-\Delta}(1)\bar G$ solves \eqref{intstr0} (for the use
  of $R ^{+}_{-\Delta}(1) = \lim  _{\varepsilon \to 0^+} R  _{-\Delta}(1+\im \varepsilon )$
  see the remark  on p. 30 \cite{SW3}) .
Therefore, it is reasonable to set
\begin{equation}\label{intstr-1}
   h = - |z|^2   zR ^{+}_{-\Delta}(1)\overline{G} +g
\end{equation}
     and think $g$ is a remainder.  In fact \eqref{intstr-1} is a normal form transformation
     intended to eliminate the term $|z|^2   z  \overline{G}$ from the r.h.s.
     in  \eqref{intstr0}. $g$ satisfies an analogous equation as $h$, but with
     a higher degree polynomial in $(z,\overline{z})$, and so it is smaller than $h$ (this goes back to \cite{BP2,SW3} )

\noindent When we substitute \eqref{intstr-1} in \eqref{intstr11} and  we  ignore   $g$,  we get
\begin{equation*}    \begin{aligned} &
\im \dot z
=  z   -  2  |z|^4 z
\int _{\R^2}
  {G}  R ^{+}_{-\Delta}(1)\overline{G}  dx  -     |z|^4 z
\int _{\R^2}
  \overline{{G}}   R ^{-}_{-\Delta}(1) {G} dx   .
\end{aligned}  \end{equation*}
We recall that  $R ^{\pm}_{-\Delta}(\lambda ) = P.V. (-\Delta- \lambda) ^{-1} \pm \im \pi \delta (-\Delta- \lambda)$  for any $\lambda >0$, where  $\delta$ is the Dirac delta function and $P.V.x^{-1}$ is the Cauchy principal value.
These can be given sense using the Fourier transform.
Multiplying by $\overline{z}$ and taking imaginary part we get
\begin{equation} \label{intstr12}  \begin{aligned} &
 \frac{d}{dt} |z|^2= - 2\pi \mathfrak{c}  |z|^6 \text{  with }
\mathfrak{c} =\int _{\R^2}
  {G}   \delta (-\Delta  -1) \overline{G}  dx \ge 0   .
\end{aligned}  \end{equation}
We conclude that $\mathfrak{c}\ge 0$ by the following formula, see ch.2 \cite{friedlander}:   \begin{equation} \label{eq:termc}  \begin{aligned} &   \mathfrak{c} =   \frac{1}{2 } \int _{|\xi |= 1} |\widehat{G } (\xi
)|^2 d\sigma (\xi ).
\end{aligned}\end{equation}
 This is the 2nd power structure discussed above.   In the context of the study of ground states of the NLS, the analogue of this formula has been  proved in  \cite{Cu2}, see formulas  \eqref{eq:FGR8}--\eqref{eq:FGR81} and   later  in this section.

 \noindent The next step
is to ask whether   in \eqref{eq:termc} we have not only  $\mathfrak{c}\ge 0$ but rather the strict inequality    $\mathfrak{c} >0$.   Obviously, for generic
$G$ we have $\left . \widehat{G} \right | _{|\xi |= 1} \neq 0$, and so
  $\mathfrak{c} >0$. Numerical computations  for random choices of   $ G $
would    yield convincingly this fact.  The strict inequality  $\mathfrak{c} >0$ and the equation of $z$ in \eqref{intstr12} yield the explicit formula
\begin{equation*}
    |z(t)| ^2= \frac{|z(0)| ^2}{(1+4 \pi \mathfrak{c} |z(0)| ^4 t) ^{\frac{1}{2}}} .
\end{equation*}
On the other hand, $ h $  will scatter by linear mechanisms, thanks also to the fact that
the forcing term   $ |z|^2   z  \overline{G}$   in  \eqref{intstr0}  is in $L^2$ for time.
(Combining this fact with Kato smoothing estimates such as Lemmas \ref{lem:lem3.3} and \ref{lem:lem3.4}, we can show $h$ has finite Strichartz norm, which implies the scattering).

\noindent Proceeding differently and as a reference for later comments, one could  integrate \eqref{intstr12}  and  write
 \begin{equation} \label{intstr120}  \begin{aligned} &
   |z(t)|^2+  2\pi \mathfrak{c} \int _0^t |z(s)|^6  ds = |z(0)|^2  .
\end{aligned}  \end{equation}
Our FGR hypothesis, stated explicitly in
\eqref{eq:FGR},   is the same   as assuming  $\mathfrak{c} >0$ in  model \eqref{model1}.
The  analogue  of $\mathfrak{c} \ge 0$    instead is rigorously proved in \eqref{eq:FGR81}.
Numerical computations are likely to prove \eqref{eq:FGR} true for generic equations exactly in the
same way they would show that  $\mathfrak{c} >0$ in the above model.

We now discuss the FGR, still  heuristically, for a model closer to the one necessary
to examine equation \eqref{NLS}. The discussion is more complicated than for model \eqref{model1},
but will yield similar conclusions.

First of all, by a Noetherian reduction of coordinates
related to the $U(1)$ invariance
   of the NLS, we will see that we  reduce  to an effective  Hamiltonian     of the form, for appropriate finite sums,
\begin{equation}   \label{model2}      \begin{aligned} &
  \mathcal{H} (z,h )= -\sum _{j=1}^{\mathbf{n}}
s_j \lambda _j |z_j|^2   + \langle (-\Delta +\omega +  \mathfrak{V})  h,      \sigma _1 h \rangle  +   \sum  _{
 |\lambda  \cdot (\mu   -\nu )| > \omega
  }   z ^\mu
 \overline{ {z }}^ { {\nu} }    \langle  G_{\mu \nu } ,\sigma _3\sigma _1h \rangle  ,\end{aligned}
\end{equation}
where: $(z,h)\in \C ^{\mathbf{n}}\times H^1(\R^2;\C^2)$ with $\overline{h}=\sigma _1h $;
the inner product $\langle  f  , g   \rangle = \int \  ^t f (x) g(x) dx$  is a bilinear map;
$\im \lambda _j \in \im \R _+$ are eigenvalues of the linearization  $ \mathcal{L}_{\omega}$;
$ \lambda = (\lambda _1, ..., \lambda _{\mathbf{n}})$; $ \mathfrak{V} (x)$ is  smooth in $x$,
 rapidly convergent to 0 as $x\to \infty$ and  is for every $x$ a  self adjoint $2\times 2$ matrix;
$\omega>0$ and in the application we are thinking it is in $\mathcal O$ of (H4);
$\mu,\nu$ are multi-index such as $\mu=(\mu_1,\cdots,\mu_{\mathbf{n}})$ and $z^\mu=z_1^{\mu_1}\cdots z_{\mathbf{n}}^{\mu_{\mathbf{n}}}$ and similar for $\bar z^\nu$;
the second (finite) sum is taken for multi-indices $\mu,\nu$;
   $ G _{\mu \nu}(x)$ is Schwartz class in $x$ with values in 2 components vectors;
  \begin{equation}\label{sigma3}
\sigma _1=\begin{pmatrix}  0 &
1  \\
1 & 0
 \end{pmatrix} \, , \
\sigma _3=\begin{pmatrix}  1 &
0  \\
0 & -1
 \end{pmatrix}  ;
\end{equation}  the number
$-s_j$ is the Krein signature of each $\im \lambda _j$, with at least one $s_j=1$
 (for the connection between
  Krein signature and number $s_j$ see in the proof of Lemma \ref{lem:bif}).

\noindent Here   $\overline{h} =\sigma _1 {h} $, the Hamiltonian  $\mathcal{H} (z,h )$  is real valued and so  $\overline{G}_{\mu \nu }  =-\sigma _1 G_{\nu \mu }$.
We consider the symplectic form     \begin{equation*}   \sum _{j=1}^{\mathbf{n}} \im \ s_j \ dz_j\wedge d\overline{z}_j +\im \langle dh  , \sigma _3
\sigma _1 dh   \rangle .
\end{equation*}
The equations are then of the form
 \begin{equation}\label{ions1}
\begin{aligned} &
  \im  \dot h  =  \mathcal{K}   h  + \sum  _{
 |\lambda  \cdot (\alpha  -\beta )| > \omega
  }   z ^\alpha
 \overline{ {z }}^ { \beta }           G _{\alpha \beta}    \text{ for }  \mathcal{K} =\sigma _3  (-\Delta +\omega +  \mathfrak{V}) \\&   \im s_j \dot z _j
=-s_j\lambda _j z_j +    \sum  _{
 |\lambda  \cdot (\mu   -\nu )| > \omega
  }  \nu _j\frac{z ^\mu
 \overline{ {z }}^ { {\nu} } }{\overline{z}_j} \langle  G_{\mu \nu } ,\sigma _3\sigma _1h \rangle   .
\end{aligned}
\end{equation}
Setting    $h =- \sum   z ^\alpha
 \overline{ {z }}^ { \beta  } R ^{+}_{ \mathcal{K} }(\lambda  \cdot (\beta    -\alpha)) G _{\alpha \beta} +g$  like in   \eqref{intstr-1}, substituting and ignoring $g$ (as we did earlier)  we get
 \begin{equation*}
\begin{aligned} &
   \im s_j \dot z _j
=-s_j\lambda _j z_j -    \sum  _{\substack{
 |\lambda  \cdot (\mu   -\nu )| > \omega
  \\ |\lambda  \cdot (\alpha   -\beta )| > \omega }}  \nu _j\frac{z ^\mu
 \overline{ {z }}^ { {\nu} }  z ^\alpha
 \overline{ {z }}^ { \beta  }   }{\overline{z}_j} \langle  G_{\mu \nu } ,\sigma _3\sigma _1R ^{+}_{ \mathcal{K} }(\lambda  \cdot (\beta    -\alpha)) G _{\alpha \beta}   \rangle   .
\end{aligned}
\end{equation*}
As we will see, we can ignore the terms where $\lambda  \cdot (\mu   -\nu )\neq \lambda  \cdot (\alpha   -\beta )$. Furthermore, up to smaller terms that we ignore, we have
\begin{equation*}
\begin{aligned} &
   \im s_j \dot z _j
=-s_j\lambda _j z_j -\sum _{L>\omega} \ \sum  _{\substack{
  \lambda  \cdot  \nu   = \lambda  \cdot  \alpha =L }}  \nu _j\frac{z ^\alpha
 \overline{ {z }}^ { {\nu} }    }{\overline{z}_j} \langle  G_{0 \nu } ,\sigma _3\sigma _1R ^{+}_{ \mathcal{K} }(-L ) G _{\alpha 0}   \rangle   .
\end{aligned}
\end{equation*}
Let us write formally $R ^{+}_{ \mathcal{K} }(-L) =P.V. \frac{1}{\mathcal{K} + L} +\im \pi \delta (\mathcal{K} + L)$  (there is a distorted Fourier transform that allows to make sense of this).
Then, using $G_{0 \nu }= -\sigma _1 \overline{G}_{  \nu 0}$,
\begin{equation}\label{eq:heu1}
\begin{aligned}
     \partial _t \sum _{j=1}^{\textbf{n}} s_j \lambda  _j
 | z _j|^2 & =  - \pi \sum _{L>\omega}  L
\Re  \left    \langle
\sum _{\substack{  \lambda  \cdot   \nu   =L } }   \overline{z }^ { \nu  }  G_{0 \nu } , \sigma _3\sigma _1  \delta (\mathcal{K} + L) \sum _{\substack{  \lambda  \cdot   \alpha   =L } } z ^{ \alpha }  G _{\alpha 0}  \right  \rangle  \\& =  - \pi \sum _{L>\omega}  L
         \langle  \overline{G}_L (z),  \sigma _3\delta (\mathcal{K} + L) {G}_L (z)  \rangle \text{ where } {G}_L (z):=  \sum _{\substack{  \lambda  \cdot   \alpha   =L } } z ^{ \alpha }  G _{\alpha 0} .
\end{aligned}
\end{equation}
Furthermore, there exists ${G}_L  ^{(0)}(z)$  s.t.
\begin{equation*}
\begin{aligned}
      \langle  \overline{G}_L (z), \sigma _3\delta (\mathcal{K} + L) {G}_L (z)  \rangle  =       \langle  \overline{G}_L  ^{(0)} (z),  \begin{pmatrix} \overbrace{\delta ( -\Delta +\omega   + L)} ^{0} &
0  \\
0 & -\delta (  \Delta -\omega   + L)
 \end{pmatrix}      {G}_L  ^{(0)}  (z)  \rangle.
\end{aligned}
\end{equation*}
Observe now that $\delta (  \Delta -\omega   + L) =  \delta (  -\Delta     -( L-\omega ) )$
and that integrating  in \eqref{eq:heu1} we get
\begin{equation}\label{eq:heu2}
\begin{aligned} &
       \sum _{j=1}^{\textbf{n}} s_j \lambda  _j
 | z _j(t)|^2   - \pi \sum _{L>\omega}  L
      \int _0^t\langle  \overline{( {G}_L  ^{(0)} (z(s)))}_2 ,  \delta (   -\Delta     -( L-\omega ))
       {( {G}_L  ^{(0)} (z(s)))}_2)  \rangle ds \\&=\sum _{j=1}^{\textbf{n}} s_j \lambda  _j
 | z _j(0)|^2.
\end{aligned}
\end{equation}
Notice that inside the integral  we have a similar  2nd power  structure
(for each $L$ we have a positive quadratic form in the vector $( z ^{\alpha} ) _{\lambda \cdot \alpha =L}$ )
to the \eqref{eq:termc} we found in    model \eqref{model1}. In particular, if  $\mathbf{n}=1$ in
\eqref{model2}, then the time integral in \eqref{eq:heu2} will be similar (possibly with different power) to the time integral in \eqref{intstr120}. So formula \eqref{eq:heu2} is a generalization of \eqref{intstr120}.
When
  the Krein signatures are all positive, that is if $s_j\equiv -1$, then
 \eqref{eq:heu2}  can be used to prove    $z(t) \stackrel{t\to \infty}{\rightarrow} 0$  (and so the asymptotic stability of the standing wave). Indeed, starting from a
 $|z(0)|$   small,  \eqref{eq:heu2} is telling us that also
  $|z(t)|$  remains small. Furthermore, the fact that the integrals remain bounded
  as $t\to \infty$  can be used, along with the fact that $|\dot z(t)|$ remains bounded (which can be seen by the 2nd equation in \eqref{ions1}),  to show that  $z(t) \stackrel{t\to \infty}{\rightarrow} 0$.
  However, to make this rigorous we need to have something  analogous to the inequality $\mathfrak{c}>0$, which is exactly the meaning of \eqref{eq:FGR}.

If  there are negative Krein signatures, and so some $s_j=1$, obviously the
proof of  $z(t) \stackrel{t\to \infty}{\rightarrow} 0$ breaks down since
in the l.h.s. of \eqref{eq:heu2}  there are terms with different signs whose
size could become large even if the sum is small.

However, if    we know that a solution $u(t)$ remains close to
  a standing wave, and consequently the corresponding $|z(t)|$ remains  small, then
  \eqref{eq:heu2} allows to prove   $z(t) \stackrel{t\to \infty}{\rightarrow} 0$ because we know that the sum  $\sum _{j }  s_j \lambda  _j
 | z _j(t)|^2$ remains small. This allows to control the integrals and, with
 hypothesis (H13), that is  \eqref{eq:FGR}, to conclude
  $z(t) \stackrel{t\to \infty}{\rightarrow} 0$.

The validity of  \eqref{eq:FGR}  (or of more explicit formulations
of the formula) is an open question.
Proofs  related to special situations are in
\cite{bambusicuccagna,Komech,adami}.
It is fair to expect that numerical
experiments will confirm strict  positivity for most examples, like for
the $\mathfrak{c}$  in \eqref{eq:termc}. Some numerical verifications are in \cite{CKP,KPS},
but there is room for more systematic studies. These are absent in the literature
not because some intrinsic difficulty, but simply because the ideas
in \cite{Cu2,Cu3} are not well known.

Having given a general overview of the main concepts and results of this paper,
we list now the content of the remaining section of the paper. In Section \ref{sect:statement}
we state hypotheses  and give in Proposition \ref{prop:asstab} a statement
that is more detailed than     {Theorem} \ref{th:asstab}.  In Sections \ref{section:set up}--\ref{sec:pfprop}
we describe the proof of  Proposition \ref{prop:asstab}.
The proof
is basically the same of the proofs of asymptotic stability  of ground states of the NLS
in \cite{Cu0,Cu2}.
In Section \ref{section:set up}, after introducing a  natural system
of coordinates related to the \textit{modulation} we state in {Proposition} \ref{prop:darboux}  a result
on Darboux coordinates, whose proof is in \cite{Cu0} and which is a key step to the subsequent search of an effective Hamiltonian.
We will only give a sketch here.
  The expansion in these new coordinates of the functional
$K(u)$, defined in \eqref{eq:eqK}, is again given without proof since the proof is in \cite{Cu0}. In Section \ref{sec:speccoo} we complexify  $  L^2(\R^2, \R ^2)$
and consider the spectral decomposition of the linearization
$\mathcal{L} _{\omega  }$ in  $  L^2(\R ^2, \C ^2)$. This produces discrete
coordinates $z=(z_1,..., z _{\mathbf{n}})\in \C ^{\mathbf{n}}$ and a continuous coordinate  $f\in L^2(\R ^2, \R ^2)$.
We then state the expansion of the functional
$K(u)$ in these coordinates, with the proofs in \cite{Cu0}.
Next step is the search
of an effective Hamiltonian by means of a Birkhoff normal forms argument.
This is accomplished in  in {Proposition} \ref{prop:Bir1}, which we give only a sketch of  the proof,  for which we refer again to
\cite{Cu0}.    {Proposition} \ref{thm:mainbounds}  contains Strichartz and smoothing estimates
satisfied by $f$, and the statement of $z(t) \stackrel{t\to \infty}{\rightarrow} 0$. We then
show that {Proposition} \ref{thm:mainbounds} implies Proposition \ref{prop:asstab}.  By
an elementary continuation argument  {Proposition} \ref{thm:mainbounds}  is in turn a consequence
of  Proposition \ref{prop:mainbounds}, whose proof is in
Section \ref{sec:pfprop}.   In consists first in  Strichartz and smoothing estimates
for $f$. We then split $f$ like in   \eqref{intstr-1}  as a term dependent only of $z$, which is the part
of $f$ which really affects the $z$'s, and a $g$ which is smaller, does not affect $z$ substantially  and can be treated as a reminder term.   The estimates on $f$ and  on $g$ are the same
of \cite{CT}. Finally in   Section \ref{sec:pfprop} we return to the Fermi golden rule,
explaining why $z(t) \stackrel{t\to \infty}{\rightarrow} 0$.  All the  estimates
are proved in the literature, for example in \cite{CM1}.
Therefore, we limit ourselves at describing the structure of the argument.
However we give a sketch of the proof for some important theorem (especially Darboux theorem and Birkhoff normal forms arguments) for reader's convenience.

In Section \ref{sec:applications} we discuss the cubic quintic equation \eqref{NLS3-5}.
We discuss how starting from the numerical observations in \cite{pego}  it satisfies
hypothesis (H14) for values $\omega >\omega _{cr}$
close to $ \omega _{cr}$. Since our theory does not apply to
translation invariant equations like \eqref{NLS3-5} we show that by adding a small
radial potential with a non degenerate  local  minimum  at 0 produces spectrally
stable vortices which still satisfy  (H14) because their linearization is a small
perturbation of that of    \eqref{NLS3-5}.  At the end of   Section \ref{sec:applications}
we also discuss the status of the other of the other hypotheses for equation \eqref{NLS3-5}
perturbed by adding a liner potential.  Some of them  follow from the computations in \cite{pego}, the others  ought to be checked numerically and in out opinion are plausibly true.


 \section{Hypotheses and statements} \label{sect:statement}

To begin with,  assume the following hypotheses.

\begin{itemize}
\item[(H1)] $\beta  (0)= 0$, $\beta\in C^\infty(\R,\R)$.
\item[(H2)] There exists a $p\in \R $ such that for every
$k\ge 0$ there is a fixed $C_k$ with
\begin{equation}\label{eq:hyph2}
   \left| \frac{d^k}{dv^k}\beta(v^2)\right|\le C_k
|v|^{p-k-1} \quad\text{if $|v|\ge 1$}.
\end{equation}
\item[(H3)] $V(x)$ is smooth, non zero,  real valued,  and for any multi
index $\alpha $ there are $C_\alpha >0$ and $a_\alpha >0$ such that
 $|\partial ^\alpha _x V(x)|\le
C_\alpha e^{-a_\alpha |x|}$.\end{itemize}

For   $n \ge 1$  and   $K=\R , \C$ then   $\Sigma _n=\Sigma _n(\R ^2, K^2   )$
is the   Banach space   with
 \begin{equation}\label{eq:sigma}
\begin{aligned} &
      \| u \| _{\Sigma _n} ^2:=\sum _{|\alpha |\le n}  (\| x^\alpha  u \| _{L^2(\R ^2   )} ^2  + \| \partial _x^\alpha  u \| _{L^2(\R ^2   )} ^2 )  <\infty .      \end{aligned}
\end{equation} We set $\Sigma _0= L^2(\R ^2, K^2   )$.  We   define $\Sigma _{t}$   by  $
      \| u \| _{\Sigma _t}  :=  \|  ( 1-\Delta +|x|^2)   ^{\frac{t}{2}} u \| _{L^2}   <\infty    $ for $t\in \R$.
 For $t\in \mathbb{N}$ the two definitions are equivalent, see \cite{Cu4}.

\begin{itemize}
\item[(H4)]
There exists an open interval $\mathcal{O}\subset \R_+$ such that
\begin{equation}
  \label{eq:B}
  \Delta u-Vu-\omega u+\beta(|u|^2)u=0\quad\text{for $x\in \R^2$},
\end{equation}
admits a  function  $\omega \to \phi _ {\omega } $ which for any $k$ is in $C^2( \mathcal{O} , \Sigma _k(\R ^2, \C )) $. We also assume that  $ \phi _{\omega } (x) =e ^{\im m \arg (x)   } \psi _{\omega } (|x|)$  like in \eqref{vortex}.
\item [(H5)] We have for $q(\omega ):=  2^{-1}  \| \phi _ {\omega }\|^2_{L^2(\R^2)}$,   see \eqref{eq:charge},
\begin{equation}
  \label{eq:1.2}
  q'(\omega )    \neq 0
 \text{  for all $\omega\in\mathcal{O}$.}\end{equation}

\end{itemize}

\begin{remark}\label{rem:sym}
  Notice that by a standard bootstrapping argument  we can relax hypothesis (H4) by only asking that  $\omega \mapsto \phi _ {\omega } $ be in $C^1( \mathcal{O} , H^1(\R ^2, \C )) $.
\end{remark}
We  identify $\C =\R^2$  setting $w_1=\Re w$  and  $w_2=\Im w$ for $w\in \C$ . In particular  we  identify the imaginary unit $\im $
 with $-J$
  where
\begin{align}\label{eq:J}
J=\begin{pmatrix}  0 & 1  \\ -1 & 0
 \end{pmatrix},
\end{align}
 and the bound state $\phi_\omega$ with $\begin{pmatrix} \Re \phi_\omega \\ \Im \phi_\omega \end{pmatrix}$.

 We consider the   \textit{strong}  symplectic form defined
\begin{equation} \label{eq:Omega} \Omega (X,Y) =  \langle  J^{-1} X , Y \rangle .
\end{equation}

 \begin{definition}\label{def:ext}
 We denote by $\langle  \ ,\ \rangle $  also the bilinear
  form
 in  $L^2(\R ^2, \C ^2)$ obtained extending \eqref{eq:hermitian}.  We extend  $\Omega$  to $L^2(\R ^2, \C ^2)$  as a bilinear    form.
\end{definition}
For $F\in C^1( \mathbf{U},\R)$ with  $\mathbf{U}$ an open subset  of $H^1(\R ^2, \R^2)$,  the gradient $\nabla F (u)$
is defined by
$  \langle \nabla F (u), X\rangle = dF(u)  X $, with $dF(u)$ the  Frech\'et derivative at $u$. If $F\in C^2(\mathbf{U},\R ) $ it remains defined    $\nabla ^2F(u)\in C^0(\mathbf{U},B(H^1(\R ^2, \R^2), H^{-1}(\R ^2, \R^2)))$ (with $B(\mathbb{X},\mathbb{Y})$    the space of $\R $--linear bounded operators from a Banach space $\mathbb{X}$ to another a Banach space $\mathbb{Y}$).

    The $\phi _\omega   $ are
constrained critical points of $ {E} $ with associated Lagrange
multiplier  $-\omega$ so that  $
 \nabla  {E} (\phi _\omega   )= -\omega  \nabla  Q (\phi _\omega   )   $.
  		The \textit{linearization} of the NLS at $\phi_\omega$  is
 \begin{align} &
 \mathcal {L}_\omega:=   J(\nabla ^2
 E(\phi _\omega   )+\omega  )=  J(-\Delta +\omega  +V) +J \mathcal{V}_{\omega}  \text{ with} \label{eq:linearizationL}  \\&   \mathcal{V}_{\omega} :=  \begin{pmatrix}  \beta (|\phi _\omega  |^2  )  + 2 \beta '(|\phi _\omega  |^2  ) (\Re \phi _\omega )^2  & \beta '(|\phi _\omega  |^2  )   (\Re \phi _\omega ) (\Im \phi _\omega )    \\  \beta '(|\phi _\omega  |^2  )   (\Re \phi _\omega ) (\Im \phi _\omega )  &   \beta (|\phi _\omega  |^2  )  + 2 \beta '(|\phi _\omega  |^2  ) (\Im \phi _\omega )^2
 \end{pmatrix}  .  \nonumber
 \end{align}
    Since there is a natural identification $L^2(\R ^2, \C ^2)=  L^2(\R ^2, \R ^2)\otimes _{\R} \C $  and $H^2(\R ^2, \C ^2)=  H^2(\R ^2, \R ^2)\otimes _{\R} \C $, the operator
  $\mathcal {L}_\omega$  extends naturally in a
  $\C$ linear operator in  $L^2(\R ^2, \C ^2)$  with domain $H^2(\R ^2, \C ^2)$ simply starting from formulas $ \mathcal {L}_\omega (v\otimes _\R z)=(\mathcal {L}_\omega  v) \otimes _\R z $.

  Starting from  $ \overline{v \otimes _\R z}  := v \otimes _\R \overline{z} $ a   complex conjugation can be defined in   $    L ^ 2(\R ^2, \R ^2)\otimes _\R \C$, which should not be confused with the complex conjugation in the initial $    L ^ 2(\R ^2, \C )$.
  Using this complex conjugation we consider in $    L ^ 2(\R ^2, \C )$ the
  hermitian
  form  $\langle  f ,\overline{g} \rangle $, where we recall that  $\langle  \ ,\  \rangle $.
is a bilinear form.

 \noindent  Notice that if $\mathcal {L}_\omega \xi  =z   \xi $ with $z \in \C$, then applying this   complex conjugation  we obtain
$\mathcal {L}_\omega \overline{\xi}   =\overline{z}  \overline{\xi } $.

 \noindent  Notice also  that as operators in $L ^ 2(\R ^2, \R ^2)$  we have
$J \mathcal {L}_\omega = -\mathcal {L}_\omega  ^* J $  and
  $  \mathcal {L}_\omega J= -J\mathcal {L}_\omega  ^*   $.
  This extends also for the operators in $L ^ 2(\R ^2, \C ^2)$.
From this we conclude  that $\sigma ( \mathcal {L}_\omega)$  is symmetric
  with respect to the coordinate axes.

We consider only   standing waves which are \textit{spectrally
  stable}.
  Specifically, we will assume
  \begin{itemize}
\item[(H6)]    $\sigma  (\mathcal {L}_\omega  ) \subset  \im \R $  for all $\omega \in \mathcal{O}.$
\end{itemize}
Since (H4) implies that $\phi _\omega (x)$ is exponentially decreasing
to 0 as $x\to \infty$, we know that always for the \textit{essential spectrum}  we have $\im \sigma  _e (\mathcal {L}_\omega  )= (-\infty , -\omega ]\cup [\omega , \infty )$.    Hence (H6) is all about
the set of eigenvalues $\sigma _p(\mathcal {L}_\omega )  $. We want  our $\phi _\omega  $'s  to be
also  \textit{linearly
  stable}. This is somewhat ambiguous as in principle it should  mean  that $\| e^{t\mathcal {L}_\omega}\| _{L^2\to L^2}$ is bounded, which is never true since $\mathcal {L}_\omega$ is not skew--adjoint
  and has a nontrivial Jordan block at 0. By linear stability we mean (H6) and
  the following two additional hypotheses.
 \begin{itemize}
\item[(H7)]    For  $N_g ( L ) :=\cup _{j=1}^\infty \ker (L^j)$  we have \begin{equation} \label{eq:kernel1}\begin{aligned}
&\ker {\mathcal L}_\omega  =\text{Span}\{ J     \phi _\omega     \} \text{ and } N_g ( {\mathcal L}_\omega ) = \text{Span}\{ J \phi _\omega , \partial _{ \omega}    \phi _\omega  \} .
\end{aligned}
\end{equation}

\item[(H8)] For any eigenvalue $\mathbf{e} \in \sigma _p( {\mathcal L}_\omega )  \backslash \{ 0 \}$  the algebraic and geometric dimensions
    coincide.
\end{itemize}
Notice that $N_g ( {\mathcal L}_\omega )$  for \eqref{NLS3-5}  has been computed in \cite{pego}.  For what happens when a potential is added to  \eqref{NLS3-5} breaking
the translation invariance see Lemma \ref{lem:zhou}.

We   assume the following hypotheses.
 \begin{itemize}
\item[(H9)]   There are no eigenvalues  of ${\mathcal L}_\omega$
contained in $\sigma  _e (\mathcal {L}_\omega  )$.

\item[(H10)] The points $\pm \im \omega $ are neither eigenvalues not
resonances of $\mathcal {L}_\omega$, i.e.   if  ${\mathcal L}_\omega F
 =\pm \im \omega F$ in a distributional sense  for an $F\in L^\infty (\R ^2)$, then $F=0$.
\end{itemize}
As we mentioned in the introductions oth conditions appear to be generically true.

\noindent We have the symmetry       $J  \mathcal {L}_\omega =- \mathcal {L}_\omega ^*J $.
 Thus, \eqref{eq:kernel1} implies
\begin{align}\label{eq:Lstar}
N_g (\mathcal{L}_\omega ^{\ast})  =\text{Span}\{   \phi _\omega ,
   J^{-1}\partial _{\omega}    \phi _\omega   \} .
\end{align}
We have the following
 beginning of Jordan blocks decomposition, where we use the hermitian from   $\langle  f ,\overline{g} \rangle $ to define the orthogonal spaces,  \begin{align} 	
\label{eq:begspectdec2}& L^2(\R ^2, \C ^2)= N_g(\mathcal{L}_\omega)\oplus N_g^\perp
(\mathcal{L}_\omega ^{\ast}) \  .
\end{align}
     \cite{CT}, as    consequence of \cite{kato}, proves that  $\| e^{t\mathcal {L}_\omega}  |_{N_g^\perp
(\mathcal{L}_\omega ^{\ast})} \| _{L^2\to L^2}$ is bounded  under (H6)--(H10).
 Appropriate dispersive and Strichatz estimates  can be proved for
 the restriction  $e^{t\mathcal {L}_\omega}\left | _{L^2_c({\omega } )}\right .$ for the space $L^2_c({\omega } )$, see Lemma \ref{lem:Specdec} below.

We assume existence of non zero eigenvalues.
\begin{itemize}
\item [(H11)]  For any $\omega \in \mathcal{O} $
there is  a number   $\mathbf{n}\ge 1$ and positive numbers  $0<\lambda
_1   \le \lambda  _2   \le ...\le \lambda  _\mathbf{n}    $
such that $\sigma _p(\mathcal{L} _{\omega })$ consists exactly of the numbers $\pm \mathbf{e}_j $  and 0, where we set   $  \mathbf{e}_j (\omega ):= \im \lambda  _j (\omega)$. We assume that
there are fixed integers $\mathbf{n}_0=0< \mathbf{n}_1<...<\mathbf{n}_{l_0}=\textbf{n}$
such that
$\lambda  _j   = \lambda _i   $ exactly for $i$ and $j$
both in $(\mathbf{n}_l, \mathbf{n}_{l+1}]$ for some $l\le l_0$. In this case $\dim
\ker (\mathcal{L} _{\omega }  -\im \lambda  _j (\omega)  )=\mathbf{n}_{l+1}-\mathbf{n}_l$.
 We denote by  $N_j\in \mathbb{N}$  the number such that $ N_j+1 =\inf  \{ n\in \N : n \im \lambda  _j  \in \sigma _e(\mathcal{L} _{\omega }) \}  $. We set $\mathbf{N}=\sup _j N_j$.

\end{itemize}
Notice that in (H11) we do not ask any more uniformity with respect to $\omega$, as in \cite{Cu2,Cu4,CM1,CM2}. We remark that the fact that the sum of all the algebraic dimensions
of the eigenvalues of $\mathcal{L} _{\omega }$  is finite can be proved    from  (H10) and   from the fact that   each $\phi _\omega (x)$ is in fact not only a Schwartz function
in $\mathcal{S}(\R ^2, \C)$ but converges exponentially to 0  as $x\to \infty$, see \cite{IW}. The proof is standard, is similar to an argument in p.305 \cite{rauch} and  involves extending the resolvent  beyond the resolvent set as
a meromorphic function. Since we are in dimension 2, the discussion of what happens near $\pm \im \omega$  is more complicated than the 3D argument near 0 in \cite{rauch}, but nonetheless an accumulation of eigenvalues near $\pm \im \omega$ can be excluded using $\pm \im \omega$.

 \noindent The following is a rather standard non--degeneracy hypothesis in the context of normal forms arguments.
\begin{itemize}

\item[(H12)] For distinct  $\lambda  _{j_1}  <...<\lambda  _{j_k}  $
    and    $\mu\in \Z^k$ with
  $|\mu| \leq 2 \mathbf{N}+3$, then  $$
\mu _1\lambda  _{j_1}  +\dots +\mu _k\lambda  _{j_k}  =0 \iff \mu=0\ .
$$
\end{itemize}
It is plausible that (H12) is generically true.

\noindent Next we assume the Fermi golden rule which we will state explicitly later and on which we commented at length in the Introduction.
\begin{itemize}
\item [(H13)]
The Fermi golden rule  Hypothesis (H13)   in Section
\ref{sec:pfprop}, see \eqref{eq:FGR}, holds.
\end{itemize}

So far the hypotheses
(H1)--(H13) are similar to the analogous ones in \cite{Cu2}. In  \cite{Cu2}
though the main result is that the standing waves are (asymptotically) stable,
while here we want to prove instability, that is the opposite. So we need an hypothesis which will generate orbital instability. To obtain this hypothesis we    consider  the signature, or Krein signature,
see \cite{kollar, kollar1}.

 Recall the extensions of $\langle \ , \ \rangle $ and
 $\Omega $  in $L^2(\R ^2, \C ^2)$  made in Def. \ref{def:ext}. Recall that  $\sigma _p  (\mathcal{L} _{\omega } )=\sigma _p  (\mathcal{L} _{\omega } ^*)$.
By general argument we have the following result.
\begin{lemma}
  \label{lem:Specdec} The following spectral decomposition remains determined:
\begin{align}  \label{eq:spectraldecomp} &
L^2(\R ^2, \C ^2) = L^2_d(\mathcal{L} _{\omega } )\oplus L^2_c(\mathcal{L} _{\omega } )\text{  where   $L^2_c({\omega } )=(L^2_d(\mathcal{L} _{\omega } ^* ))^\perp$ and     for $L=\mathcal{L} _{\omega },\mathcal{L} _{\omega }*$ }\\&
L^2_d(L ) :=
N_g  (L  )  \oplus  \widetilde{L}^2_d(L ) \text{  with }   \widetilde{L}^2_d(L ) :=\oplus _{\mathbf{e} \in
\sigma _p(\mathcal{L} _{\omega })\backslash \{ 0\}}   \ker (L  - \mathbf{e}
 )   . \nonumber
\end{align}
\end{lemma}
We denote by $P_c(\omega)$ the projection on $L^2_c({\omega } ) $ associated to \eqref{eq:spectraldecomp}.

The form $\Omega$  remains symplectic also   in $ \widetilde{L}^2_d(\mathcal{L} _{\omega } )$.
The proof of the following lemma is elementary, see for example  Lemma 5.2 \cite{Cu0}.
  \begin{lemma}
  \label{lem:basis}   For any $\omega  \in \mathcal{O}$ and corresponding $\mathbf{n}$
  in (H11) there are functions $\xi _j(\omega  )\in \Sigma _k  $ for any $k$ and $j=1,...,\mathbf{n}$  such that the following facts hold.
  \begin{itemize}
\item[(1)] $\xi _j(\omega )\in \ker (\mathcal{L}_{\omega} -\im \lambda _j (\omega  )) $ for all $j$.

 \item[(2)]  $\Omega ( \xi _j  {(\omega )} , {\xi} _k  {(\omega )})
 =0$ for all $j$ and $k$ and
  $ \Omega (  \xi _j  {(\omega)} , \overline{{\xi}} _k  {(\omega )} ) =\im  s_j \delta _{jk} $  with $s_j \in \{ 1 ,-1   \} $.
\end{itemize}
\end{lemma}
\qed

In the case of ground states, that is when $\phi _\omega (x)=  \psi _{\omega } (|x|)$
with $\psi _{\omega } (|x|)>0$  and $m=0$, then $s_j\equiv -1$, see Lemma 2.7  \cite{CM2}. Here,
where $m\neq 0$, we assume
instead what follows.
   \begin{itemize}
\item[(H14)]    There exists at least one $j$ s.t. $s_j=1$.
\end{itemize}
We have already discussed, and we will say more in  Section \ref{sec:applications},  that
it has been shown  numerically  that this hypothesis occurs
for spectrally stable vortices of the cubic--quintic NLS \eqref{NLS3-5}.

  {Theorem} \ref{th:asstab}  is a consequence of the following proposition,
     which
  is  a consequence of   \cite{Cu0,Cu2,CM1,CT}.
\begin{proposition}\label{prop:asstab}  Let    $\omega_{1}\in \mathcal{O}$ and assume (H1)--(H13). Then there
exist constants $\epsilon _0>0$  and $C_0$  s.t. if  $u\in C^0([0,\infty ), H^1)$ is a solution of \eqref{NLS}
 with $  \sup _{t\ge 0}\inf _{\vartheta \in \R }  \| u (t) -e^{\im \vartheta}\phi _{\omega_{1} }\| _{H^1}<\epsilon <\epsilon _0$   then  there exist  $\omega
_+\in\mathcal{O}$, $\theta\in C^1(\R ;\R)$ and $h _+ \in H^1$
with $\| h_+\| _{H^1}+|\omega _+ -\omega_1|\le C_0 \epsilon $
such that
\begin{equation}\label{scattering}
\lim_{t\to   \infty}\|u(t )-e^{\im \theta(t)}\phi_{\omega
_+}-e^{\im t\Delta }h _+\|_{H^1}=0 .
\end{equation}
It is possible to write $u(t,x)=e^{\im \theta(t)}\phi_{\omega  (t)}
+ A(t,x)+\widetilde{u}(t,x)$  with $|A(t,x)|\le C_N(t) \langle x
\rangle ^{-N}$ for any $N$, with $\lim _{t\to \infty }C_N(t)=0$,
with $\lim _{ t \to  \infty } \omega  (t)= \omega _+$, and such
that for any admissible  pair $(q,p)$, i.e.
\begin{equation}\label{admissiblepair}  1/q+1/p= 1/2\,
 ,
q> 2,
\end{equation}
we have
\begin{equation}\label{Strichartz} \|  \widetilde{u} \|
_{L^q ( [0,\infty ),W^{1,p} )}\le
 C_0\epsilon .
\end{equation}
 \end{proposition}

From Section \ref{section:set up} to Section \ref{sec:pfprop} we focus on  {Proposition} \ref{prop:asstab}. There are various steps. We aim at showing that there exists an effective Hamiltonian of the form \eqref{model2}. This has to be found   through a Birkhoff normal forms argument, see Theorem 11 \cite{hofer}. In order to initiate the process we need to
to find  and appropriate system of Darboux coordinates.

\section{Modulation and Darboux coordinates} \label{section:set up}

Asymptotic (or conditional asymptotic) stability arguments require traditionally, since
 \cite{SW1}, the choice
of appropriate    \textit{modulation} coordinates.
Indeed, since we are discussing the stability of vortices, it is natural to
express a solution $u(t)$ which is close to a vortex as a sum of a vortex plus an error and   to frame stability in terms of what happens to the error. This and more is what modulation aims to do.
The first step to define precisely this vaguely stated aim  is   the following      standard Modulation Lemma.

 \begin{lemma} [Modulation Lemma]
  \label{lem:modulation}
     Fix $\underline{n} \in \Z$, $\omega_1\in \mathcal O$ and $\Psi _1=e^{-J \vartheta _1}\phi _{\omega _1}$, where $\mathcal O$ is given in (H4) and $J$ is given in \eqref{eq:J}. Then there exists a neighborhood $\U _{\underline{n}  } $   of $\Psi _1$    in $ \Sigma _{-\underline{n}}(\R ^2, \R ^{2 }) $
    and   functions $\omega \in C^\infty (\U _{\underline{n}}   , \mathcal{O})$
	and $\vartheta \in C^\infty (\U _{\underline{n}}   , \R   )$ s.t.  $\omega (\Psi _1)=\omega _1 $   and $\vartheta
(\Psi _1) =\vartheta _1$
  and s.t.  $\forall u\in \U _{\underline{n}}   $
\begin{equation}\label{eq:ansatz}
\begin{aligned}
  &
 u =   e^{-J \vartheta } (  \phi _{\omega } +R)
 \text{  and $R\in N^{\perp}_g (\mathcal{L}_\omega ^*)$.}
\end{aligned}
\end{equation}
\end{lemma}

We give a sketch of the proof.
See also Lemma 2.4  \cite{Cu0} or Lemma 2.2 \cite{CM1} for a detailed proof.

\proof
It suffices to apply the implicit function Theorem to
\begin{align*}
F_1(\vartheta,\omega,u)=\begin{pmatrix}\<e^{J\vartheta}u-\phi_\omega, \phi_\omega\>\\ \<e^{J\vartheta}u-\phi_\omega, J^{-1}\partial_\omega \phi_\omega\>\end{pmatrix}.
\end{align*}
$\left.\frac{\partial F_1}{\partial(\vartheta,\omega)}\right|_{(\vartheta,\omega,u)=(\vartheta,\omega_1,e^{-J\vartheta}\phi_{\omega_1})}$ will be invertible because of (H5).
 \qed

The above Modulation Lemma is the starting point to find appropriate coordinates in the neighborhood of $\Psi_1$ in $H^1(\R ^2, \C )$.   Solutions $u(t)$ starting close to $\Psi_1$ will admit
a time dependent decomposition \eqref{eq:ansatz}. If $u(t)$ stays close to the orbit of
$\Psi_1$ for all time and scatters to a vortex, this will be equivalent at showing that $R(t)$ scatters to 0 as $t\to \infty$ and $\omega (t) \stackrel{t\to +\infty  }{\rightarrow} \omega _+$ for some $\omega _+\in \mathcal{O}$. This will be proved working on the parameters in
the r.h.s. of \eqref{eq:ansatz}.

\noindent Equation \eqref{NLS}  can be expressed as $\dot u =J \nabla E(u)$.  The following
discussion is standard, and is only sketched  in order to give an overview of the use
of the parameters $(\vartheta , \omega , R)$.
By substituting
the r.h.s. of  \eqref{eq:ansatz} and using $\nabla E(\phi _\omega )=- \omega \phi _\omega$
we obtain after standard computations
\begin{equation}\label{eq:linear0}
    ( \omega -\dot \vartheta   ) J\phi _\omega + \dot   \omega  \partial _{\omega} \phi _\omega +\dot R = \mathcal{L}_{\omega} R +N(R) -  ( \omega -\dot \vartheta   ) JR,
\end{equation}
where $N(R)$ contains terms which are  quadratic or higher order for $R$ small.
  Denote
by  $P_{N_g }(\omega) =P_{N_g(\mathcal{L}_\omega)}$   the projection on
$N_g(\mathcal{L}_\omega) $ related to \eqref{eq:begspectdec2}.
Because $\<\phi_\omega,\partial_\omega\phi_\omega\>=q'(\omega)$, we have
\begin{equation} \label{eq:projNg} \begin{aligned}
 P_{N_g }(\omega)X= -(q'(\omega ))^{-1}\(\Omega(X,\partial_\omega \phi_\omega) J
  \phi _\omega\  +\Omega(J\phi_\omega,X) \partial
_{\omega}\phi _\omega \),  \quad \forall \, X\in \mathcal{S} '(\R ^2, \C^2) .
\end{aligned}
\end{equation}
Applying the projection
\begin{equation} \label{eq:projP} \begin{aligned} & P(\omega):=1-P_{N_g }(\omega). \end{aligned}
\end{equation}
to \eqref{eq:linear0} we obtain for $R$ the following equation
\begin{equation}\label{eq:linear1}
   \dot R = \mathcal{L}_{\omega} R +P(\omega)N(R) -  ( \omega -\dot \vartheta   ) P(\omega)JR.
\end{equation}
 Since, as is well known, the term $( \omega -\dot \vartheta   )$  is higher order in $R$,
 we can think of \eqref{eq:linear1} as a perturbation of $ \dot R = \mathcal{L}_{\omega} R$.
 It is natural now to look at the rest of the spectrum of $\mathcal{L}_{\omega} $. The main difficulty is to show that the discrete components of $R$ associated to the point spectrum
 of $\mathcal{L}_{\omega} $, which at a linear level want to oscillate like the $e ^{-\im t}z(0)$ component
 of the linearization of \eqref{intstr0}--\eqref{intstr11},  will lose their energy
 because of some friction originating from the nonlinear interaction with the continuous components of $R$. This effect will be captured by an argument similar to the Fermi golden rule discussed in the Introduction. For that argument to work we need to find an appropriate system of coordinates.

Lemma \ref{lem:modulation} does not provide coordinates.  We co back to
the projection $P_{N_g }(\omega)$. We have $\Omega(P_{N_g}(\omega)X,Y)=\Omega(X,P_{N_g}(\omega)Y)$.
 By (H4)--(H5) and \eqref{eq:kernel1}  for $\mathcal{S}(\R^2, K^2)= \cap _{k\ge 0}\Sigma _k(\R^2, K^2)$   the space of  Schwartz functions and for  $\mathcal{S}'(\R^2, K^2)= \cup _{k\le 0}\Sigma _k(\R^2, K^2)$ the space of  tempered distributions,
    we have  \begin{equation} \label{eq:projreg} \begin{aligned} & P_{N_g }(\omega)\in C^{\infty}
(\mathcal{O}, B (\mathcal{S}'(\R ^2, K^2),\mathcal{S}(\R ^2, K^2))) \text{  for $K=\R , \C$}. \end{aligned}
\end{equation}
For $P(\omega )$ defined as in \eqref{eq:projP}    we have $  \omega  \to  P(\omega )P( {\omega}_1) \in C^\infty (\mathcal {O} ,B (\Sigma _k ,\Sigma _k))
$ for any $k$.
By (H4)  we have $P(\omega)\stackrel{\omega\to \omega_1}{\rightarrow} P(\omega  _1)$  in the operator topology of $B(\Sigma_k,\Sigma_k)$.
Thus, writing $$P(\omega)P(\omega_1)=(1+(P(\omega)-P(\omega_1)))P(\omega_1),$$ we see that
  there exists an $a>0$ such that
 if
 $|\omega -\omega _1|<a$   the  map
$P(\omega )P(\omega _1)  $   restricts into an isomorphism
from $N_g^\perp (\mathcal{L}_{\omega _1 }^*)\cap   X_k$  to $   N_g^\perp (\mathcal{L}_{\omega }^*)  \cap X_k$
for any $k\ge - {\underline{n}}$ with  $X_k$ equal either to   $H^k$ or to  $ \Sigma _k$. Hence
     for   $k\ge - {\underline{n}}$ the
map
\begin{equation} \label{eq:coordinate} \begin{aligned} &
  \R   \times \{ \omega : |\omega -\omega _1  |<a  \} \times ( N_g^\perp(\mathcal{L}_{\omega_1}^*) \cap  X_k)  \to X_k  , \\&  (\vartheta , \omega , r) \to u=   e^{-J
\vartheta }
  (\phi _{\omega} + P(\omega)  r),
\end{aligned} \end{equation}
is for  $\| r\| _{ X _k}<a$  a local $C^\infty$ diffeomorphism in the image.
Therefore, $(\vartheta , \omega , r)$ in \eqref{eq:coordinate} provides an initial system
of independent coordinates.

If we consider the function $Q=Q(u)$, the map $(\vartheta , \omega , r)\to (\vartheta , Q , r)$ is a local diffeomorphism because of the assumption (H5).
Indeed, applying implicit function theorem to
\begin{align}\label{eq:coorQro}
F_2(Q,\rho,r,\omega)=Q(\phi_{\omega}-P_{N_g}(\omega)r)+\rho+\<\phi_{\omega}-P_{N_g}(\omega)r, r\>-Q,
\end{align}
there exists $\omega(Q,\rho,r)$ which is a smooth function defined in the neighborhood of $(Q(\phi_{\omega_0}),0,0)\in \R\times \R\times \Sigma_{-\underline{n}}$.
Notice that $F_2(Q,Q(r),r,\omega)=Q(\phi_\omega+P(\omega)r)-Q$.
We have put an auxiliary variable $\rho$ because if we directly put $\omega(Q,r)$ to be the implicit function of $\tilde F_2=Q(\phi_\omega+P(\omega)r)-Q$, then we will only able to define $\omega$ in the neighborhood of $(Q(\phi_{\omega_0}),0)$ in $\R\times L^2$.
Differentiating, \eqref{eq:coorQro} by $Q,\rho,r$, we have
\begin{align}
\partial_Q \omega&=-\partial_{\rho}\omega=A^{-1},\label{eq:pqo}\\
\<\nabla_r \omega,X\>&=A^{-1}\(\<P(\omega)r,P_{N_g}(\omega)X\>+\<P_{N_g}(\omega)r,X\>\),\label{eq:pqr}
\end{align}
where $A=\<\phi_{\omega}+P(\omega)r,\partial_\omega \phi_\omega+\partial_\omega P(\omega)r\>$.

We now expand $\Omega$ by using the coordinates $(\vartheta,Q,r)$.
Notice that $$X=du X = \partial_{\vartheta} u X_\theta + \partial_Q u X_Q + \<\nabla_r u, X_r\>,$$ where $X_\vartheta=d\vartheta X, X_Q=dQ X$ and $X_r = dr X$.
Then, after some cancelations, we obtain
\begin{align}
\Omega=&-d\vartheta\wedge dQ + \Omega(P(\omega)dr,P(\omega)dr)\label{eq:Oexpand}\\
&+A^{-1}dQ\wedge \Omega(\partial_\omega P(\omega)r,P(\omega)dr)+\<\nabla_r \omega+\partial_q \omega r,dr\>\wedge \Omega(\partial_\omega P(\omega) r, P(\omega)dr).\nonumber
\end{align}
Notice that the coordinates $(\vartheta,Q,r)$ are not a system of Darboux coordinates
for the symplectic form $\Omega$.

We now prepare some notations.
\begin{itemize}
\item
Let $F$ be a Frech\'et differentiable function.
Then, its hamiltonian vector    field $X_F $ is defined by  $ \Omega (X_F ,Y)=  dF (Y)$ for any given
  vector  $Y.$ In particular, we have $X_F=J\nabla F$.
\item
For    $F,G  $  two scalar valued functions,
we set the   Poisson bracket by
$
  \{ F,G \}  :=  dF (X_G ).
$
\item
If $\mathcal{G}  $   has values in a given
Banach space $\mathbb{E}$ and $G$ is a scalar valued function, then we set  $
  \{ \mathcal{G} ,G \}  :=  d\mathcal{G}  (X_G )$.
\end{itemize}

 In the coordinate system $(\vartheta ,Q, r )$  our NLS can be expressed as
\begin{align}\label{eq:hamE}
\dot Q= \{Q,E\}=(X_E)_Q,\quad \dot \vartheta = \{\vartheta,E\}=(X_E)_\vartheta,\quad \dot r= \{r,E\}=(X_E)_r.
\end{align}

Further, comparing the coefficients of $Y_\vartheta $ in $\Omega(X_E,Y)=dEY$ by \eqref{eq:Oexpand}, we have $(X_E)_Q=dQ X_E=0$.
Therefore, we have $\dot Q=0$.
Notice that this shows that with coordinates $  (\vartheta , Q , r)$  we have achieved
a \textit{reduction of order} in the system,  see \cite{olver} p. 412,  effectively reducing to the variable $r$ only.

\noindent In the
sequel we  choose  $  \omega _0$  such that if
$u_0$ is the initial value in  \eqref{NLS}, then
\begin{equation}\label{eq:p0}
 Q (\phi _{\omega_0})=Q (u_0).
\end{equation}
We consider (recall $q(\omega )= Q   ( \phi _{\omega })$)
\begin{align}\label{eq:eqK}
 {K} ( u):&= {E}  (u)-  {E} \left (  \phi _{\omega_0}\right
 ) +  \omega (u)   ( Q (u)   - q  (  {\omega_0})) .
\end{align}
Then, since $(X_Q)_\vartheta=-1,(X_Q)_Q=0$ and $(X_Q)_r=0$, we see that \eqref{eq:hamE} is equivalent to
 \begin{equation*}       \begin{aligned} &  \dot Q=0 \, , \quad  \dot \vartheta =\{ \vartheta  , K   \} +\omega \, , \quad \dot r=   \{ r , K   \} .
\end{aligned}
\end{equation*}
 See \cite{Cu0} Lemma 2.6 and Section 2.3, and it is important that $Q(u_0)= q(\omega _0)$.

In the sequel the changes of coordinates  will differ from the identity transformation
by perturbations that can be written in terms  of the two classes of symbols  which we introduce now.

		\begin{definition}\label{def:scalSymb} For $I$  an interval with 0 in the
interior,  $\mathcal{A}\subset  \R \times   \R \times (N_g^\perp(\mathcal{L}_{\omega_1}^*) \cap   \Sigma_{-n})$
  a neighborhood of $(q(\omega _1),   0 ,0)$,
we  say that   $ \mathfrak{F} \in C^{m}(I\times \mathcal{A},\R)$
 is $\mathcal{R}^{i,j}_{n, m}$
 if    there exists    a $C>0$   and a smaller neighborhood
  $\mathcal{A}'$ of  of $(q(\omega _1),   0 ,0)$
  s.t.
 \begin{equation}\label{eq:scalSymb}
  |\mathfrak{F}(t, Q ,   \varrho , r )|\le C \|  r\| _{\Sigma   _{-n}}^j (\|  r\|
  _{\Sigma   _{-n}}+|\varrho | +|Q  -q(\omega _1)  |)^{i} \text{  in $I\times
  \mathcal{A}'$}.
\end{equation}
 We  will write also  $\mathfrak{F}=\mathcal{R}^{i,j}_{ n,m}$ or
 $\mathfrak{F}=\mathcal{R}^{i, j}_{ n,m} (t, Q ,   \varrho , r )$.
Given a function $F:I\times \U _{\omega _1}\to \R$ for $\U _{\omega _1}$
a neighborhood of $\phi _{\omega _1} $ in   $L^2(\R ^2, \R^2)$, we say
that $F=\mathcal{R}^{i,j}_{ n,m}$ if there is  a $\mathcal{R}^{i,j}_{ n,m}$
function s.t. $F(t, u ) = \mathcal{R}^{i, j}_{ n,m} (t, Q ,   Q(r) , r )$.

We say  $\mathfrak{F}=\mathcal{R}^{i, j} _{n, \infty}$  if $\mathfrak{F}=\mathcal{R}^{i,j}_{n, m}$ for all $m $.
We say    $F=\mathcal{R}^{i, j}_{\infty, m} $       if   we can take
$n$ arbitrarily large. If   $F=\mathcal{R}^{i, j}_{\infty, m} $ for any $m$, we set $F=\mathcal{R}^{i, j}  $.
\end{definition}

\begin{definition}\label{def:opSymb}  A    $\mathfrak{T} \in C^{m}(I\times
\mathcal{A},\Sigma   _{n}  (\R^2, \C ^{2 }))$,  with $I\times \mathcal{A}$
like above,
 is $ \mathbf{{S}}^{i,j}_{n,m} $   and  we  write as above
 $\mathfrak{T}= \mathbf{{S}}^{i,j}_{n,m}$  or   $\mathfrak{T}= \mathbf{{S}}^{i,j}_{n,m} (t, Q ,   \varrho , r )$,
 if      there exists  a $C>0$   and a smaller neighborhood
  $\mathcal{A}'$ of  $(p_0, p_0  , 0  )$   s.t.
 \begin{equation}\label{eq:opSymb}
  \|\mathfrak{T}(t, Q ,   \varrho , r )\| _{\Sigma   _{n}}\le C \|  r\| _{\Sigma   _{-n}}^j (\|  r\|
  _{\Sigma   _{-n}}+|\varrho | +|Q  -q(\omega _1)  |)^{i}  \text{  in
  $I\times \mathcal{A}'$}.
\end{equation} We  use notation
$\mathfrak{T}=\mathbf{{S}}^{i,j}_{n,\infty }$, $\mathfrak{T}=\mathbf{{S}}^{i,j}_{\infty,m}$    and  $\mathfrak{T}=\mathbf{{S}}^{i,j} $
as
above. As above, given a function $T:I\times \U _{\omega _1}\to \Sigma   _{-n}$ we write  $F=\mathbf{{S}}^{i,j}_{ n,m}$ if there is  a $\mathbf{{S}}^{i,j}_{ n,m}$
function s.t. $T(t, u ) =\mathbf{{S}}^{i, j}_{ n,m} (t, Q ,   Q(r) , r )$.

\end{definition}

Next we consider the following symplectic form: \begin{equation} \label{eq:Omega0} \Omega _0  := - d \vartheta
\wedge dQ    +  \Omega    (dr  \ ,  dr \ )
 .
\end{equation}
This is how our symplectic $\Omega $  form should look in appropriate coordinates. Indeed in   Section 3 \cite{Cu0}  the following Darboux Theorem is proved.\begin{proposition}[Darboux Theorem]
  \label{prop:darboux}  There is a local diffeomorphism $\mathcal{F}$ around
  $\phi _{\omega _1}$ in $L^2(\R ^2, \R^2)$   such that  $\mathcal{F}^*\Omega  = \Omega _0$
  and
  which in
  the $(\vartheta , Q, r) $ coordinates is of the form
  \begin{equation} \label{eq:darb1}
    \begin{aligned} &
    \vartheta '=\vartheta + \mathcal{R}^{0,2}( Q ,Q(r),  r)\, , \quad Q '=Q  \, , \\&
    r'  =e^{J\mathcal{R}^{0,2}( Q ,Q(r),  r) } ( r+  \mathbf{{S}}^{1,1}( Q ,Q(r),  r) ).
\end{aligned} \end{equation}
 \end{proposition}
\qed

\begin{remark}\label{rem:omegas}
  Notice   that the idea of taking as fixed point $\phi _{\omega _1}$  rather than $\phi _{\omega _0}$ as in  \cite{Cu0},  is taken from \cite{bambusi}.   The proof  of Prop. \ref{prop:darboux}  is unaffected.
\end{remark}

For the convenience of the readers we give a sketch of the proof here.

\begin{proof}[Sketch of the proof of Proposition \ref{prop:darboux}]
To make a change of coordinate to convert the symplectic form $\Omega$ into $\Omega_0$ we need three steps.
First, we find a $1$-form $\Gamma$ s.t.\ $\Omega-\Omega_0=d\Gamma$.
Next, we solve $i_{\mathcal X^s} \Omega_s = -\Gamma$, where $i_X \omega(Y):=\omega(X,Y)$ and $\Omega_s:=\Omega_0+s(\Omega-\Omega_0)$.
Finally, let $\mathcal Y_s$ be the flow of $\frac{d}{ds}\mathcal Y_s = \mathcal X^s(\mathcal Y_s)$.
Then, we have
\begin{align*}
\frac{d}{ds}\(\mathcal Y_s^* \Omega_s\)=\mathcal Y_s^*\(\mathcal L_{\mathcal X^s} \Omega_s -\partial_s \Omega_s\)=\mathcal Y_s^* (d i_{\mathcal X^s} \Omega_s - d \Gamma)=0.
\end{align*}
Thus, $\mathcal Y:=\mathcal Y^1$ gives us the desired transformation.
This is a standard proof of Darboux theorem (see \cite{hofer}).

\noindent In our situation, we have to care about the regularity of the transformation
(or in other words, error from the identity).
Therefore, we need to compute $\Gamma$ rather explicitly.

First, we seek   $\Gamma$. It suffices to find some $\Gamma$ satisfying $\Gamma=B-B_0+dC$, where
\begin{align*}
2B_0&=Qd\vartheta + \Omega(r,dr),\\
2B&= \Omega(u,\cdot)\\&=Q d\vartheta + \Omega(P(\omega)r,dr)+\Omega(\phi_\omega,P(\omega)dr)+ \Omega(\phi_\omega+P(\omega)r,\partial_\omega \phi_\omega +\partial_\omega P(\omega) r) d \omega.
\end{align*}
It is elementary that $dB=\Omega$, $dB_0=\Omega_0$.
We have
\begin{align*}
2(B-B_0)=&d\( \Omega(\phi_\omega,P(\omega)r)\) + \Omega(P(\omega)r,\partial_\omega P r) \partial_Q \omega dQ\\&+ \Omega(-P_{N_g}(\omega)r+\Omega(P(\omega)r,\partial_\omega P(\omega)r)\partial_\rho \omega Jr + \Omega(P(\omega)r,\partial_\omega P(\omega)r)\nabla_r \omega,dr) .
\end{align*}
Therefore, we can choose $\Gamma=\Omega(\Gamma_r,dr)+\Gamma_Q dQ$ as
\begin{align*}
2\Gamma_r&=-P_{N_g}(\omega)r+\Omega(P(\omega)r,\partial_\omega P(\omega)r)\partial_\rho \omega Jr + \Omega(P(\omega)r,\partial_\omega P(\omega)r)\nabla_r \omega,\\\
2\Gamma_Q&=\Omega(P(\omega)r,\partial_\omega P r) \partial_Q \omega.
\end{align*}
Since $P_{N_g}(\omega) r=P_{N_g}(\omega)\(P(\omega_0)-P(\omega)\)r$ and $|\omega-\omega_0|\sim Q(r)$, we see
$\Gamma_r = \mathbf S ^{1,1}+\mathcal R^{0,2} Jr$ and $\Gamma_Q=\mathcal R^{0,2}$.

\noindent Next, we solve $i_{\mathcal X^s} \Omega_s =-\Gamma$. Since $s\(\Omega(\mathcal X^s,\cdot)-\Omega_0(\mathcal X^s,\cdot)\)$ can be handled as a perturbation, the main part of the equation will be $\Omega_0(\mathcal X^s,\cdot)=-\Gamma$.
Therefore, we have $\mathcal X^s_Q=0$, $\mathcal X^s_\theta=\mathcal R^{0,2}$ and $\mathcal X^s_r=\mathcal S^{1,1}+\mathcal R^{0,2}Jr$.
Finally, solving $\frac{d}{ds}\mathcal Y_s = \mathcal X^s(\mathcal Y_s)$, we have the conclusion.
\end{proof}

It is well known that normal forms processes are based on Taylor expansions
of the Hamiltonian, see \cite{hofer}. So we need an expansion of the functional $K$ defined
in \eqref{eq:eqK} in terms of the Darboux coordinates. This is provided by the following,
proved in   Lemma 4.3 \cite{Cu0}.
\begin{lemma}
  \label{lem:back} Consider an integer $L\in \N $  s.t. $L>p$ for $p$ the exponent \eqref{eq:hyph2} in  hypothesis (H2).
  For any preassigned $(k,m)$ and in the coordinates $( Q, \vartheta , r)$ of \eqref{eq:darb1}, $K$ admits the  expansion
\begin{align}     &
  d(\omega )- d(\omega _0) - (\omega - \omega _0)q(\omega _0)+   \frac  {1}{2}\Omega (  \mathcal{L}_\omega P(\omega )r,P(\omega ) r )
   +  \resto ^{1,2} _{k,m}   ( Q ,Q(r) , r)   +E_P(P(\omega )r)
+\textbf{R} ^{\prime \prime},
\nonumber \\&
\textbf{R} ^{\prime \prime }:=     \sum _{d=2} ^{L-1}
\langle B_{d } ( Q ,Q (r)  , r ),  (P(\omega)r)  ^{   d} \rangle
      +\int _{\mathbb{R}^2}
B_L (x, r(x),  Q ,Q (r)  , r     )   (P(\omega)r)^{   L}(x) dx \text{ with:} \label{eq:back1}
\end{align}
\begin{itemize}
\item   $ d(\omega ) = E (\phi _\omega ) +\omega q(\omega )$;
  \item   $B_2(q(\omega _0), 0,0  )=0$;

\item $(P(\omega)r)^d(x)$  represent $d-$products of components of $P(\omega)r$;

  \item
$B_{d
}( \cdot ,  Q ,\varrho  , r  ) \in C^{m } ( \U _{-k},
\Sigma _k (\mathbb{R}^2, B   (
 (\mathbb{R}^{2  })^{\otimes d},\mathbb{R} ))) $  for $2\le d \le 4$ with $\U _{-k}\subset \Ph ^{-k}$ a neighborhood of   $(q(\omega _1),   0 ,0)$ in $  \R \times   \R \times (N_g^\perp(\mathcal{L}_{\omega_1}^*) \cap   \Sigma_{-n})$;
 \item  for
$ \zeta \in \mathbb{R}^{2   }$
 and $( Q, \varrho ,r) \in \U _{-k}$
 we have  for $i+j\le m$
\begin{equation} \label{eq:B5}\begin{aligned} &  \|  \partial _r^j\partial _{  \zeta,Q ,\varrho     }
^iB_L(\cdot ,\zeta,  Q ,\varrho  , r  ) \| _{B  ( \Sigma_{-k} ^{\otimes j} ,  \Sigma _k(\mathbb{R}^2,   B   (
 (\mathbb{R}^{2   })^{\otimes L},\mathbb{R} ))} \le C_i .
 \end{aligned}  \end{equation}
\end{itemize}
\end{lemma}
\begin{remark}\label{rem:clarif} We have $d(\omega )- d(\omega _0) - (\omega - \omega _0)q(\omega _0) =O(\omega - \omega _0) ^2 =  \resto ^{2,0}(Q (r) )  + \resto ^{1,2} _{k,m}   ( Q ,Q (r) , r)$.

\noindent In  Lemma 4.3 \cite{Cu0} inequality \eqref{eq:B5} is stated for $|\zeta  |\le \varepsilon$
for some small $\varepsilon >0$, but in fact in the proof is unnecessary, thanks to  (H2).  Notice also that  $L=5$ in \cite{Cu0}, but a similar  proof holds for our choice of $L$.
\end{remark}

\begin{proof}[Sketch of the proof]
Notice that here we are just expanding $K(u)=S_{\omega}(u)-S_{\omega}(\phi_{\omega_0})$ such as
\begin{align*}
K(\phi_\omega+P(\omega)r)=S_{\omega}(\phi_\omega)+\<\nabla^2 S_{\omega}(\phi_{\omega})P(\omega)r,P(\omega)r\>+o\(\(P(\omega)r\)^2\)-d(\omega_0)+(\omega-\omega_0)q(\omega_0),
\end{align*}
where we have used $\nabla S_\omega(\phi_\omega)=0$.
\end{proof}

If we expand $P(\omega)r=r + (P(\omega)-P(\omega _1))r$  we obtain what follows, see Lemma 4.4 \cite{Cu0}.
 \begin{lemma}
  \label{lem:back11} The expansion of $K$ in  Lemma \ref{lem:back}
  can be rewritten as follows,with similar notation:
\begin{align}     &
  K=\resto ^{2,0} _{k,m} (Q,Q (r) ) +  {2}^{-1}\Omega (  \mathcal{L} _{\omega _1}  r,  r )   +  \resto ^{1,2} _{k,m}   ( Q ,Q (r) , r)   +E_P( r)+\textbf{R} ^{\prime  },
\nonumber
\\&   \textbf{R} ^{\prime   }:=     \sum _{d=2} ^{L-1}
\langle B_{d } ( Q ,Q (r)  , r ), r  ^{   d} \rangle
      +\int _{\mathbb{R}^2}
B_L (x, r(x),  Q ,Q (r)  , r     )  r^{   L}(x) dx \text{ .} \label{eq:back1}
\end{align}
 \end{lemma}

\section{Spectral coordinates   associated to $\mathcal{L} _{\omega _1}$}
\label{sec:speccoo}

Recall that   $r\in N_g^\perp (\mathcal{L}_{\omega _1 }^*)\cap  L^2(\R ^2, \R ^2)$.
We consider the spectral decomposition of $r$ in terms of $\mathcal{L} _{\omega _1}$:
  \begin{equation}
  \label{eq:decomp2}\begin{aligned}&
  r  =\sum _{j=1}^{\mathbf{n}}z_j  \xi
  _j  (\omega _1 )  +
\sum _{j=1}^{\mathbf{n}}\overline{z}_j \overline{\xi  } _{  j} (\omega _1 )   + f    \text{  where   $f\in L _{c}^2(\omega _1)$ and also $f\in L^2 (\R ^2, \R ^2)$.}  \end{aligned} \end{equation}
This yields new coordinates $r\to (z,f)$ to replace $r$.

 \noindent Correspondingly we have the expansion
\begin{equation} \label{eq:decompOm}  \Omega    (dr \ ,  dr \ ) =\sum _{j=1}^{\mathbf{n}} \im \ s_j \ dz_j\wedge d\overline{z}_j +\Omega    (df \ ,  df \ ) .
\end{equation}
Equation   $\dot r =    \{ r , K   \} $ splits into
\begin{equation}\label{eq:systf}
 \im \dot z_j =  s_j  \frac{\partial}{\partial \overline{z}_j}K \, , \quad \dot f =    \{ f , K   \},
\end{equation}
where we recall that $s_j\in \{ 1,-1\}$ and $s_j=1$ for at least one $j$.

We have reduced our NLS to   system  \eqref{eq:systf}.
Obviously, having replaced $r$ with $(z,f)$, we need to rewrite the expansion of $K$ in {Lemma}
  \ref{lem:back11}  in terms of   $(z,f)$. This is done in  Lemma 5.4 \cite{Cu0}, which we quote.
\begin{lemma}
  \label{lem:ExpH11}  In the coordinate system $( z,f)$
  near $(0,0)$ for any preassigned pair
	  $(k,m)$  we have an expansion
\begin{equation}  \label{eq:ExpH11} \begin{aligned} &  {K} = \resto ^{2,0} _{k,m} (Q, Q (f) )+H_2 '   +\sum _{j=-1}^{4} {\textbf{R} _{ j }} + \resto ^{1,2}_{k,m}(Q ,Q (f), f)  \text{  with   what follows.} \end{aligned} \end{equation}
 \begin{itemize}
\item[(1)]
 For $\varrho =Q (f) $
\begin{equation}  \label{eq:ExpH2} H_2 '=  -\sum _{j=1}^{\mathbf{n}}
s_j \lambda _j |z_j|^2 +  \sum _{\substack{ |\mu +\nu |=2\\
\mathbf{e}   \cdot (\mu -\nu )=0}}   \resto ^{1,0}_{k,m} (Q ,\varrho )
   z^\mu
\overline{z}^\nu   + 2 ^{-1} \Omega (\mathcal{L}_{\omega
 _1}   f,     f) .
\end{equation}

\item[(2)] We have for $\varrho =Q (f) $:
\begin{equation*}   \begin{aligned} &{ {\mathbf{R}} _{  - 1 }}=
\sum _{\substack{ |\mu +\nu |=2\\
\mathbf{e}  \cdot (\mu -\nu )\neq 0  }} \resto ^{1,0}_{k,m} (Q ,\varrho )  z^\mu
\overline{z}^\nu  +  \sum _{|\mu +\nu |  = 1} z^\mu \overline{z}^\nu  \langle
J   \mathbf{S }^{1,0}_{k,m} (Q ,\varrho ),f\rangle  ;\end{aligned}
\end{equation*} for $ {\mathbf{N}}$ as  in (H12),   $\varrho =Q (f) $,
$ g_{\mu \nu }(Q , \varrho )=\resto ^{0,0}_{k,m}(Q , \varrho )$, $ G_{\mu \nu }(Q , \varrho )=\mathbf{S} ^{0,0}_{k,m}(Q , \varrho )$ and with the symmetries $g_{\nu \mu }=\overline{g}_{\mu \nu }$  and $G_{\nu \mu }=-\overline{G}_{\mu \nu }$, we have
 \begin{equation*}    \begin{aligned} &   {\textbf{R} _0}=   \sum _{|\mu
    +\nu |= 3}^{2 \mathbf{{N}}+1} z^\mu
\overline{z}^\nu   g_{\mu \nu }(Q , \varrho ) ;  \quad {\textbf{R} _1}= \im
\sum _{|\mu +\nu |= 2}^{2 \mathbf{{N}} } z^\mu \overline{z}^\nu \langle J
G_{\mu \nu }(Q , \varrho ), f\rangle ; \\ &   {\textbf{R} _2}=
\langle \mathbf{S }^{1,0}_{k,m}(Q , \varrho ) ,   f  ^{   2} \rangle
     \text{ with $\mathbf{B}_{2 } (q(\omega _1) , 0)=0$},
\end{aligned}\nonumber \end{equation*}
  where   $f^d(x)$    represents schematically $d-$products of components of
  $f$;
   \begin{align} &   {\textbf{R} _3}=   \sum _{ \substack{
|\mu +\nu |=\\=  2N+2}} z^\mu \overline{z}^\nu  \resto ^{0,0}_{k,m}( Q   ,
z,  \varrho  , f)        +\sum _{ \substack{ |\mu +\nu |=\\= 2N+1}} z^\mu
\overline{z}^\nu  \langle J    \mathbf{S} ^{0,0}_{k,m}(Q  , z,   \varrho ,f ),
f\rangle  ;\nonumber \\ &   {\textbf{R} _4}=  \sum
_{d=2}^{L-1} \langle B_{d } (Q   ,  z ,\varrho ,f),   f  ^{   d} \rangle
       +\int _{\mathbb{R}^2}
B_L  (x, f(x),  Q , z,Q (f)  , f  )  f^{   L}(x) dx + E_P (  f),  \nonumber \end{align}
where  the $B $'s are like in Lemma \ref{lem:back}.
\end{itemize}
\end{lemma}
\qed

Now we start discussing about the normal forms argument.  It will consist in
eliminating as many terms as possible  from $\mathbf{R}_{j}$ with $j=-1,0,1$.

\noindent We set, for $\mathbf{n}$  the number associated to $\omega _1$ in (H11),
\begin{equation}\label{eq:eqvee}
    \mathbf{e} =( \mathbf{e}_1,...,\mathbf{e}_{\mathbf{n}}).
\end{equation}
Some of the monomials in $\mathbf{R}_{j}$ with $j= 0,1$  cannot be eliminated
because they are \textit{resonant}, that is of the following type.

\begin{definition}[Normal Forms]
\label{def:normal form} A function $Z(  z,  \varrho , f )$ is in normal form if  $
  Z=Z_0+Z_1
$
where $Z_0$ and $Z_1$ are finite sums of the following type:
\begin{equation}
\label{e.12a}Z_1=   \im  \sum _{\mathbf{e}      \cdot(\mu-\nu)\in \sigma _e(\mathcal{L} _{\omega _1 })}
z^\mu \overline{z}^\nu \langle  J  G_{\mu \nu}(   \varrho   ),f\rangle,
\end{equation}
where  $ G_{\mu \nu} =\mathbf{S}^{0,0}_{k,m}(\varrho)$
 for  fixed $k,m\in  \mathbb{N}$;
\begin{equation}
\label{e.12c}Z_0= \sum _{  \mathbf{e} (\omega _1)  \cdot(\mu-\nu)=0} g_{\mu   \nu}
(   \varrho  )z^\mu \overline{z}^\nu,
\end{equation}
and $g_{\mu   \nu} =  \resto ^{0,0}_{\infty ,m}(\varrho)$.
We assume furthermore that  $Z_0$ and $Z_1$ are real valued for $f\in L^2(\R ^2, \R ^2)$, and hence
    $\overline{g}_{\mu \nu}=g_{\nu\mu  }$ and $\overline{G}_{\mu \nu}= -G_{\nu \mu }$.
 \end{definition}

With an appropriate canonical change  of coordinates (that is, it preserves the r.h.s. of \eqref{eq:decompOm})   the term $\mathbf{R}_{-1 } $ and all non resonant terms in $\mathbf{R}_{0 } $ and $\mathbf{R}_{1 } $   cancel
out. Indeed we have the following fact, which we quote from Theorem 6.4 \cite{Cu0}.
   \begin{proposition} [Birkhoff normal forms]
  \label{prop:Bir1} There is a canonical transformation $ (z,f)  \stackrel{\mathcal{F} }{\rightarrow}(z',f') $
  where  \begin{equation} \label{eq:Bir11}
    \begin{aligned} &
    z '=z + \mathcal{R}^{0,2}( Q ,z,Q(f),  f)   \, , \\&
    f  =e^{J\mathcal{R}^{0,2}( Q ,z,Q(f), f) } ( f+  \mathbf{{S}}^{1,1}( Q ,z,Q(f ),  f) ).
\end{aligned} \end{equation}
   such that in the new coordinates $(z,f)$ we have
\begin{equation}  \label{eq:Bir12} \begin{aligned} &  {K}(Q,z,f)  = \resto ^{2,0} _{k,m} (Q, Q (f) )+H_2 '  +\textbf R_{0}+\textbf R_{1}+ \resto  \text{ with }\\&
 \resto
=\resto ^{1,2}_{k,m}(Q ,Q (f), f) +\sum _{j=2}^{4}
\textbf{R} _j  + \widehat{\textbf{R}} _2(  Q ,  z     ,\varrho , f), \end{aligned} \end{equation}
with $H_2 ' $ and  $\textbf{R} _j$   like in Lemma  and where we have:

\begin{itemize}
\item[(1)]  the term   $\textbf{R} _{-1}   $   in   \eqref{eq:ExpH11} is here $\textbf{R} _{-1}   =0$;

 \item[(2)]  all the nonzero terms in $\textbf{R} _0   $ with $|\mu +\nu |\le 2 \mathbf{N} +1 $  are in normal form,
that is $ \mathbf{e} \cdot (\mu -\nu )=0$, and are in $Z_0$;

\item[(3)] all the nonzero terms in $\textbf{R} _1  $ with $|\mu +\nu |\le 2 \mathbf{N}$  are in normal form,
that is $ \mathbf{e} \cdot (\mu -\nu )\in \sigma _e(\mathcal{H} _{p_0})$, and are in $Z_1$;

 \item[(4)]  we have   $\widehat{\textbf{R}} _2
 \in C^{m} ( \mathbb{{U}} ,\C
 )$   for  $\mathbb{{U}}\subset  \R \times   \C ^{\mathbf{n}} \times \R \times  P_c  \Sigma_{-k} $
  a neighborhood of $(q(\omega _1),   0 ,0,0)$ and
\begin{equation*} \begin{aligned} &    | \widehat{\textbf{R}} _2 (Q   ,z ,f, \varrho
)|  \le C (|Q - q(\omega _1) |+ |z|+  \| f \| _{\Sigma _{-k}}) \| f \| _{\Sigma
_{-k}}^2.  \end{aligned}\end{equation*}
\end{itemize}
\end{proposition}

\begin{proof}[Sketch of the proof]
Our canonical transformation will be generated by Hamiltonian functions of the following form:
\begin{align*}
\sum _{\substack{ |\mu +\nu |=m\\
\mathbf{e}  \cdot (\mu -\nu )\neq 0  }} A_{\mu,\nu} (Q ,\varrho )  z^\mu
\overline{z}^\nu  +  \sum _{\substack{|\mu +\nu |  = m-1\\\mathbf{e}      \cdot(\mu-\nu)\not\in \sigma _e(\mathcal{L} _{\omega _1 })}} z^\mu \overline{z}^\nu  \Omega(
  B_{\mu,\nu} (Q ,\varrho ), f).
\end{align*}
The Hamilton vector flow generated from this Hamiltonian vector field will be
\begin{align*}
z_j(s)&\sim z_j-s\im s_j \(\sum _{\substack{ |\mu +\nu |=m\\
\mathbf{e}  \cdot (\mu -\nu )\neq 0  }} \nu_j A_{\mu,\nu} (Q ,\varrho )  \frac{z^\mu
\overline{z}^\nu}{\bar z_j}  +  \sum _{\substack{|\mu +\nu |  = m-1\\\mathbf{e}      \cdot(\mu-\nu)\not\in \sigma _e(\mathcal{L} _{\omega _1 })}} \nu_j \frac{z^\mu \overline{z}^\nu}{\bar z_j}  \Omega(
 B_{\mu,\nu} (Q ,\varrho ),f )\),\\
f(s)&\sim f +  s \sum _{\substack{|\mu +\nu |  = m-1\\\mathbf{e}      \cdot(\mu-\nu)\not\in \sigma _e(\mathcal{L} _{\omega _1 })}} z^\mu \overline{z}^\nu   B_{\mu,\nu} (Q ,\varrho ).
\end{align*}
Thus,
\begin{align*}
 &-\sum _{j=1}^{\mathbf{n}}
s_j \lambda _j |z_j(1)|^2  + 2 ^{-1} \Omega (\mathcal{L}_{\omega
 _1}   f(1),     f(1)) \sim  -\sum _{j=1}^{\mathbf{n}}
s_j \lambda _j |z_j|^2  + 2 ^{-1} \Omega (\mathcal{L}_{\omega
 _1}   f,     f)\\&+\sum _{\substack{ |\mu +\nu |=m\\
\mathbf{e}  \cdot (\mu -\nu )\neq 0  }} \mathbf{e}  \cdot (\mu -\nu ) A_{\mu,\nu} (Q ,\varrho ) z^\mu
\overline{z}^\nu +\sum _{\substack{|\mu +\nu |  = m-1\\\mathbf{e}      \cdot(\mu-\nu)\not\in \sigma _e(\mathcal{L} _{\omega _1 })}}
z^\mu \overline{z}^\nu \Omega (\(\mathcal{L}_{\omega
 _1}  -\mathbf e\cdot (\mu-\nu)\)   B_{\mu,\nu} (Q ,\varrho ),     f).
\end{align*}
By the above, we can erase the nonresonant terms.
\end{proof}

 In  \eqref{eq:Bir12} the functional $K$ is written  in a form which is essentially the
 same of \eqref{model2}.  What follows in Sections \ref{sec:equations}--\ref{sec:pfprop}   is rather close to the classical discussion
 in \cite{BP2,SW3}. The difference between these papers and  \cite{Cu2,bambusicuccagna}
  lies in the fact that the latter two use in an essential form the Hamiltonian
  structure of the system to study higher order interactions between discrete and
  continuous modes. The ideas originate  from \cite{Cu3}.
  The method \cite{BP2,SW3} does not work well with higher order interactions and
  requires very stringent restrictions
  on the spectrum of the linearization, which are completely eased in  \cite{Cu2,bambusicuccagna}.

   As we have seen in the analysis of \eqref{model2} the 2nd power,
 the structure of the Fermi golden rule will be  easily   seen in the framework   provided by {Proposition}
  \ref{prop:Bir1}. However the informal analysis on \eqref{model2}  which we made in
  Section \ref{section:introduction}  has to be supplemented by a number of   estimates, especially for the variable $f$. So we will need to write the equation for $f$ and derive  some estimates.

\section{Equations}
\label{sec:equations}

In the new coordinates $(z,f)$ in Proposition \ref{prop:Bir1} our NLS continues to be of the form \eqref{eq:systf}. In particular  we have
\begin{equation} \label{eq:eqf1} \begin{aligned}              \dot f &=   J  \nabla _{f} \resto ^{2,0} _{k,m} (Q, Q (f) )+ J\nabla _{f}H'_2  (Q,z, Q (f),f )+ J\nabla _{f} Z_0(Q,z, Q (f)  ) \\& + J\nabla _{f} Z_1(Q,z, Q (f),f ) + J \nabla _{f} \resto  (Q,z, Q (f),f ),   \end{aligned}  \end{equation}
where \begin{equation} \label{eq:nablaH1} \begin{aligned} &
     \nabla _{f} ( \resto ^{2,0} _{k,m} (Q, Q (f) ) +  H'_2 +Z_0 )=  \mathcal{L}_{\omega _1
 }  f   + {A} '   J   f,
 \\&  {A}':= \partial  _{Q (f)} \resto ^{2,0} _{k,m} (Q, Q (f) )
     +\sum _{\substack{ |\mu +\nu |\ge 2\\
\textbf{e}  \cdot (\mu -\nu )=0}}
\partial _{Q (f)}  a_{\mu \nu} ( Q, Q  (f) )  z^\mu
\overline{z}^\nu,
		   \end{aligned}\end{equation}
and similarly  we split the 2nd line of \eqref{eq:eqf1}  into
 \begin{equation} \label{eq:nablaH12} \begin{aligned} &
       {A} ^{\prime\prime}     J   f  -  \im    \sum _{\mathbf{e}      \cdot(\mu-\nu)\in \sigma _e(\mathcal{L} _{\omega _1 })}
z^\mu \overline{z}^\nu       G_{\mu \nu}( Q, Q  (f)  )  +  J\nabla _{f} \resto  (Q,z, \varrho,f )_{| \varrho = Q (f)},
 \\&  {A} ^{\prime\prime} := \partial  _{Q (f)} [ Z_1 (Q,z, Q (f),f )
     +  \resto  (Q,z, Q (f) ,f ) ] .
		   \end{aligned}\end{equation}
So finally we write the equation of $f$ as
\begin{align}              \dot f &= \mathcal{L}_{\omega _1
 }  f  + A Jf -  \im \sum _{\mathbf{e}      \cdot(\mu-\nu)\in \sigma _e(\mathcal{L} _{\omega _1 })}
z^\mu \overline{z}^\nu       G_{\mu \nu}( Q, 0  )   + \mathfrak{R} \text{  where}\label{eq:eqf2} \\   A &=  {A} ^{\prime } + {A} ^{\prime\prime} \text{ and }\mathfrak{R} =  \nabla _{f} \resto  (Q,z, \varrho,f )_{| \varrho = Q (f)} -  \im \sum _{\mathbf{e}      \cdot(\mu-\nu)\in \sigma _e(\mathcal{L} _{\omega _1 })}
z^\mu \overline{z}^\nu    \left (   G_{\mu \nu}( Q, Q(f)  )-  G_{\mu \nu}( Q, 0  )\right ).   \nonumber \end{align}
We write the equations for $z$ as
\begin{equation}\label{eq:eqz1}
   s_j  \im \dot z_j =  \frac{\partial}{\partial \overline{z}_j} (H'_2+Z_0) + \im \sum _{\mathbf{e}      \cdot(\mu-\nu)\in \sigma _e(\mathcal{L} _{\omega _1 })}
\nu _j \frac{z^\mu \overline{z}^\nu}{\overline{z}_j} \langle  J  G_{\mu \nu}(    Q, Q(f) ),f\rangle + \frac{\partial}{\partial \overline{z}_j}\resto .
\end{equation}
We set, for $\mathbf{n}$  the number associated to $\omega _1$ in (H11),
\begin{equation}\label{eq:eqlambda}
   \lambda  =( \lambda_1 ,...,\lambda_{\mathbf{n}} ),
\end{equation}
where  the $\lambda_j= \lambda_j(\omega _1)$ are introduced in (H11).

\begin{proposition}\label{thm:mainbounds}  For $ \epsilon _0$ sufficiently small
  there exists a $C>$ s.t., given  a solution  $u (t) $  of the NLS  which satisfies
$  \sup _{t\ge 0}\inf _{\vartheta \in \R }  \| u (t) -e^{\im \vartheta}\phi _{\omega_{1} }\| _{H^1}<\epsilon <\epsilon _0$,
  for  $t\in I= [0,\infty )$  we have
\begin{align}
&   \|  f \| _{L^p_t(I,W^{ 1 ,q}_x)}  +  \|  f \| _{L^2_t(I,H^{ 1 ,-s}_x)}\le
  C \epsilon \text{ for all admissible pairs $(p,q)$,}
  \label{Strichartzradiation}
\\& \| z ^\mu \| _{L^2_t(I)}\le
  C \epsilon \text{ for all multi indices $\mu$
  with  $\lambda\cdot \mu >\omega _1  $,} \label{L^2discrete}\\& \| z _j  \|
  _{W ^{1,\infty} _t  (I )}\le
  C \epsilon \text{ for all   $j\in \{ 1, \dots ,  \mathbf{{n}}\}$, }\label{L^inftydiscrete}  \\&  \|  \omega    -  \omega _1  \| _{L_t^\infty  (I )}\le
  C \epsilon    .
  \label{eq:orbstab1}
\end{align}
Furthermore  \begin{align}
     \label{eq:asstab3}  \lim _{t\to +\infty}  z(t)=0 .
\end{align}
\end{proposition}
Notice that the case $(p,q)=(\infty , 2)$ in \eqref{Strichartzradiation} and inequalities \eqref{L^inftydiscrete}  and \eqref{eq:orbstab1} are an easy  consequence of
$  \sup _{t\ge 0}\inf _{\vartheta \in \R }  \| u (t) -e^{\im \vartheta}\phi _{\omega_{1} }\| _{H^1}<\epsilon .$

\textit{Proposition  \ref{thm:mainbounds}  implies Proposition  \ref{prop:asstab}}.
For the proof see \cite{CM1}, in particular Section 12.  First of all, by standard arguments  there is a   $f_+  \in H^1(\R ^2, \C^2)$
such that for the   $f$ in
\eqref{Strichartzradiation} and for the $A'$ in \eqref{eq:nablaH1}  we have
\begin{equation}\label{eq:scattering1}
 \begin{aligned} &    \lim _{t\to + \infty } \|   f(t) - e^{J  (\omega _1t +\int _0^t A'(s) ds)  }
  e^{ -J  t   \Delta    } f_+ \| _{H^1 }=0 ,
\end{aligned}
\end{equation}
 see Lemma 7.2  \cite{CM1}. We can express the solution $u(t) $ as  \begin{equation}\label{eq:compl0}
\begin{aligned}  &  u(t)=    e^{-J \vartheta (t) } (  \phi _{\omega(t) } +P(\omega (t))  (z'_j(t)\xi  _j+\overline{z'}_j(t)\overline{\xi}  _j +f'(t) ) \text{  where}  \\&  z '=z + \mathcal{R}^{0,2}_{k,m}( Q ,z,Q(f),  f)     \text{   and }
    f ' =e^{J\mathcal{R}^{0,2}_{k,m}( Q ,z,Q(f), f) } ( f+  \mathbf{{S}}^{1,1}_{k,m}( Q ,z,Q(f ),  f) )
 .      \end{aligned}
\end{equation}
where $(z,f)$ are the variables in Proposition  \ref{thm:mainbounds}  and $(k,m)$ are arbitrary. This follows from the fact that composing the change of variables
\eqref{eq:darb1} and \eqref{eq:Bir11}  yields a change of variables like in \eqref{eq:Bir11},
see \cite{Cu0}.

\noindent \eqref{eq:asstab3} and \eqref{eq:scattering1}   imply $\displaystyle \lim _{t   \to + \infty  }\mathbf{S}^{1,1}_{k,m} =0$ in $H^1$ and
$\displaystyle \lim _{t \to + \infty}\resto ^{0,2}_{k,m} =0$ in $\R  $.

 \noindent It is easy to see  that  when we plug \eqref{eq:compl0}  in   $\dot u =J\nabla \mathbf{E} (u)$
  we get
\begin{equation*}\label{eq:compl3}
\begin{aligned} &  \dot f=J(-\Delta  +\dot \vartheta - \dot  \resto ^{0,2}  ) f+G_1(u),
\end{aligned}
\end{equation*}
with $G_1(u)\in  C^0 (  H^1_x,   L^1_x)$.
 On the other hand,
$f$ satisfies also \eqref{eq:eqf1}, which is of the form
\begin{equation*} \label{eq:compl4} \begin{aligned} &              \dot f = J(-\Delta +\omega _1 +  A'   )
f+G_2(u)  ,    \end{aligned}  \end{equation*}
with  $G_2(u)\in  C^0 (  H^1_x,   L^1_x)$. Then   we have \begin{equation} \label{eq:compl5} \begin{aligned} &
     \kappa (u) f= G_1(u)
-G_2(u)  \text{  with }  \kappa (u) :=   \omega _1- \dot \vartheta + \dot  \resto ^{0,2} + A'   . \end{aligned}  \end{equation}
We have $  \kappa  \in C^0 (  H^1_x,   \R )$ and we claim that $\kappa =0$. Indeed, if
      $\kappa (u (t_0)) \neq 0$ for a given   solution, we can find solutions for which
$u_n (t,\cdot ) \in {\mathcal S}(\R ^2)$,  $ u_n (t_0,\cdot )\to u (t_0,\cdot )$ in $H^1(\R ^2)$,
$\|  u_n (t_0  ) \|  _{L^1(\R ^2)}\to \infty$,
$G_j( u_n (t_0  ) ) \to G_j( u  (t_0  ) )$
and $\kappa ( u_n (t_0  ) ) \to \kappa ( u  (t_0  ) )$.  This yields a contradiction
because on one hand $\|  u_n (t_0  ) \|  _{L^1(\R ^2)}\to \infty$ implies for the corresponding $f$ coordinates
$\|  f_n (t_0  ) \|  _{L^1(\R ^2)}\to \infty$, on the other hand \eqref{eq:compl5}
is telling us $ \|  f_n (t_0  ) \|  _{L^1(\R ^2)} \sim | \kappa (u (t_0))| ^{-1} \| G_1(u(t_0))
-G_2(u(t_0)) \|  _{L^1(\R ^2)}$ and so
$\|  f_n (t_0  ) \|  _{L^1(\R ^2)}\not \to \infty$.

  \noindent   Integrating  $ \omega _1- \dot \vartheta + \dot  \resto ^{0,2} + A'=0$   and by $\displaystyle \lim _{t \to + \infty}\resto ^{0,2}_{k,m} =0$ we get for a fixed $\vartheta _0\in \R $
\begin{equation*}
    \lim _{t\to +\infty} (  \omega _1 t+\int _0^t A'(s) ds- \vartheta (t)) = \vartheta _0.
\end{equation*}
Then \eqref{eq:scattering1}   for $h_+:= e^{ J  \vartheta _0  }f_+ $
 becomes
\begin{equation*}
 \begin{aligned} &    \lim _{t\to + \infty } \| e^{ -J  \vartheta (t)  }  f(t) -
  e^{ -J  t   \Delta    }  h_+ \| _{H^1 }=0 .
\end{aligned}
\end{equation*}
Using this in formula \eqref{eq:compl0}  we obtain
\begin{equation*}
 \begin{aligned} &    \lim _{t\to + \infty } \| u(t) -  e^{i \vartheta (t) }   \phi _{\omega(t) } - e^{ \im  t   \Delta    }   h_+ \| _{H^1 }=0
\end{aligned}
\end{equation*}
which yields \eqref{scattering}.   Combining  \eqref{eq:p0}  and \eqref{eq:compl0}
wee have
\begin{equation*}
 q(\omega (t) ) + Q(f) +\resto ^{0,2}_{k,m} =q(\omega _0).
  \end{equation*}
  Then
  \begin{equation*}
  \lim _{t\to + \infty } q(\omega (t) )  =q(\omega _0)  -Q(h_+) .
  \end{equation*}
Then (H5) implies that there must be a $\omega _+$ s.t. $\displaystyle \lim _{t\to + \infty }  \omega (t)   =\omega _+$.  The last sentence of {Proposition} \ref{prop:asstab}
follows from \eqref{eq:compl0} and the estimates in  Proposition   \ref{thm:mainbounds}.

\qed

By a standard continuity argument, Proposition  \ref{thm:mainbounds} is  a consequence of the following proposition.

\begin{proposition}\label{prop:mainbounds}
There exists $\epsilon _0>$ and a  constant $c_0>0$ such that   if $T>0$ and
if
  $u (t) $ is a solution  of the NLS  which satisfies
$  \sup _{t\in I}\inf _{\vartheta \in \R }  \| u (t) -e^{\im \vartheta}\phi _{\omega_{1} }\| _{H^1}<\epsilon <\epsilon _0$  where $I= [0,T] $ then,  if  the inequalities  \eqref{Strichartzradiation}--\eqref{L^2discrete}
hold  for this $I $
and for $C=C_0\ge c_0$,  they hold  also   $C=C_0/2$.
\end{proposition}

\section{Proof of Proposition \ref{prop:mainbounds} }
\label{sec:pfprop}

The proof of Proposition \ref{prop:mainbounds} is basically the same
of Proposition 6.7 in \cite{CM1}. We give a schematic description
of the main steps. The first is the following, which follows from theory
in \cite{CT} and whose proof we review in Appendix \ref{app:4.2}.

\begin{lemma}\label{lem:conditional4.2} Assume
the hypotheses of Prop. \ref{prop:mainbounds}.    Then  there is a fixed $c$ and an $s_0$
such that for all admissible pairs $(p,q)$ and all $s>s_0$
\begin{equation}
  \|  f \| _{L^p_t([0,T],W^{ 1 ,q}_x)}+ \|  f \| _{L^2_t([0,T],H^{ 1 ,-s}_x)}\le
  c _s \epsilon  + c_s   \sum _{\lambda\cdot \mu >\omega _1 }| z ^\mu  | ^2_{L^2_t( 0,T  )},
  \label{4.5}
\end{equation}
where we sum only on multiindices such that $\lambda\cdot \mu -
\lambda  _j <\omega _1$ for any  $j$ such that
for the $j$--th component of $\mu $ we have   $\mu _j\neq 0$.
\end{lemma}
\qed

The notation is simpler if we change frame.  For $M$ defined below we have
$M^{-1}
 \im J M =\sigma _3 $, see \eqref{sigma3}, and so we have
 \begin{align}   \label{eq:Homega1} &{\mathcal K}_{\omega  } :=M^{-1}
 \im {\mathcal L}_{\omega  } M  =
   \sigma _3  (-\Delta  +V+ \omega ) +  \sigma _3  M^{-1}\mathcal{V}_{\omega} M,
   \nonumber \\& \text{where }
 M:=
  \begin{pmatrix}   1  &
1  \\
-\im  &   \im
 \end{pmatrix}   \, , \quad   M^{-1} =\frac{1}{2}
  \begin{pmatrix}   1  &
\im   \\
1  &   -\im
 \end{pmatrix}   \, , \quad   \sigma _3=\begin{pmatrix} 1 & 0\\0 & -1 \end{pmatrix}
   .  \nonumber
\end{align}
 Then we set $ h=M^{-1}f$, which satisfies  for $ \mathbf{G}_{\mu \nu} :=M^{-1} {G}_{\mu \nu}  (Q,0)$  and $\mathbf{E}:=M^{-1} \mathfrak{R}$,
\begin{equation}\label{eq:eqh}
\begin{aligned}
  \im  \dot h &= \mathcal{K}_{\omega _1
 }  h  + A  \sigma _3h +  \sum _{\mathbf{e}      \cdot(\mu-\nu)\in \sigma _e(\mathcal{L} _{\omega _1 })}
z^\mu \overline{z}^\nu       \mathbf{G}_{\mu \nu}    +  \mathbf{E}  .
\end{aligned}
\end{equation}
with $A$   defined  in \eqref{eq:nablaH12}.  The last summation in \eqref{4.5}  originates from the  $z^\mu \overline{z}^\nu       \mathbf{G}_{\mu \nu} $ terms in   \eqref{eq:eqh} (or the corresponding ones in
the 1st line of \eqref{eq:eqf2}). We cancel these terms by a normal forms argument
For
\begin{equation}
  \label{eq:g variable}
g=h+ Y \,  , \quad Y:=\sum _{| \lambda  \cdot(\mu-\nu)|>\omega _1} z^\mu
\overline{z}^\nu
   R ^{+}_{\mathcal{K}_{\omega _1
 }  }  (\lambda  \cdot(\nu-\mu) )
     \textbf{G}_{\mu \nu},
\end{equation}
we have  \begin{equation} \label{eq:eq g1}  \begin{aligned} &
 \im \dot g =     \mathcal{K}_{\omega _1
 }    g   +A   \sigma _3g    + [  \im \dot Y -\mathcal{K}_{\omega _1
 }  Y ]+  \sum _{\mathbf{e}      \cdot(\mu-\nu)\in \sigma _e(\mathcal{L} _{\omega _1 })}
z^\mu \overline{z}^\nu       \mathbf{G}_{\mu \nu}  +A   \sigma _3Y +     \mathbf{E}.
\end{aligned}\end{equation}
We then compute
\begin{equation} \label{eq:eq g2}  \begin{aligned} &
   \im \dot Y = \sum _{| \lambda  \cdot(\mu-\nu)|>\omega _1}  \lambda  \cdot(\nu-\mu) z^\mu \overline{z}^\nu   R ^{+}_{\mathcal{K}_{\omega _1
 }  } ( \lambda  \cdot(\nu-\mu))
  \mathbf{G}_{\mu \nu} + \mathbf{T} \text{ where} \\&     \textbf{T} :=\sum _j \left [\partial _{z_j}Y (\im \dot z_j+ \lambda_jz_j)+\partial _{\overline{z}_j}Y (\im \dot {\overline{z}}_j- \lambda_j\overline{z}_j)
 \right ] .
\end{aligned}\end{equation}
Inserting \eqref{eq:eq g2}  in  \eqref{eq:eq g1} we obtain
\begin{equation} \label{eq:eq g3}  \begin{aligned} &
 \im \dot g =     \mathcal{K}_{\omega _1
 }    g   +A   \sigma _3g      +A   \sigma _3Y + \mathbf{T} +     \mathbf{E}.
\end{aligned}\end{equation}
So we have canceled the $z^\mu \overline{z}^\nu       \mathbf{G}_{\mu \nu} $ terms.
Notice that $\mathbf{T} $ contains terms of this type  but  by \eqref{eq:eqz1} they are smaller.
So $g$ is smaller than $h$. In fact   the following is true,  and is proved   in Lemma 4.6 in \cite{CT}  (the statement in   Lemma 4.6 \cite{CT} has a systematic typo and $L^2_t L ^{2,M}_x$ should be replaced by $L^2_t L ^{2,-M}_x$, where
$M$ there is like our $s$ here).

\begin{lemma}\label{lemma:bound g} Assume the hypotheses of Prop. \ref{prop:mainbounds}. Then   for fixed $s>1$ there  exist
a fixed $c$
such that if   $\varepsilon_0$ is sufficiently small, for any preassigned and large $L>1$  we have $\| g
\| _{L^2((0,T ), L^{2,-s}_x)}\le   c    \epsilon  $.
\end{lemma}
\qed

\noindent For $M^T $ the transpose of $M$, and using  $M^T=2 M ^{-1} $ , $f=Mh$  and $ \mathbf{G}_{\mu \nu} :=M^{-1} {G}_{\mu \nu}  (Q,0)$, by direct computation we have
\begin{equation*}
\begin{aligned}
& \im \langle  J  G_{\mu \nu}(    Q, 0 ),f\rangle = \im \langle  M^T  JM   M ^{-1} G_{\mu \nu}(    Q, 0 ),h\rangle =     2 \langle  \sigma _1   \sigma _3   \mathbf{{G}}_{\mu \nu} ,h\rangle .
\end{aligned}
\end{equation*}
Notice that $G_{\mu \nu}(    Q, 0 )=-\overline{G}_{\nu \mu}(    Q, 0 )$, see in Lemma
\ref{lem:ExpH11}, implies by $ \overline{M} ^{-1} =\sigma _1  {M} ^{-1}$
\begin{equation*}
\begin{aligned}
&    \overline{ \mathbf{{G}}}_{\nu \mu} = \overline{M} ^{-1} \overline{G}_{\nu \mu}(    Q, 0 )
=- \sigma _1  {M} ^{-1}{G}_{\mu \nu}(    Q, 0 ) =  - \sigma _1  { \mathbf{{G}}}_{\mu \nu}  .
\end{aligned}
\end{equation*}

Then substituting  \eqref{eq:g variable}  in \eqref{eq:eqz1} we obtain
\begin{equation}\label{eq:FGR0} \begin{aligned} & \im s_j \dot z _j
=
\partial _{\overline{z}_j}(H_2'+Z_0) +2   \sum  _{
 |\lambda  \cdot (\mu   -\nu )| > \omega
_1 }  \nu _j\frac{z ^\mu
 \overline{ {z }}^ { {\nu} } }{\overline{z}_j}  \langle g ,    \sigma _3   \overline{ \mathbf{{G}}}_{\nu \mu}\rangle       +
\partial _{  \overline{z} _j}  \resto  \\&    +2 \sum  _{ \substack{| \lambda   \cdot
(\alpha    -\beta )|> \omega _1
\\
 |\lambda  \cdot (\mu -\nu )|> \omega
_1 }}  \nu _j\frac{z ^{\mu +\alpha }  \overline{{z }}^ { {\nu}
+\beta}}{\overline{z}_j} \langle  R ^{+}_{\mathcal{K}_{\omega _1
 }  } ( \lambda  \cdot(\beta -\alpha))    \textbf{G}
_{\alpha \beta } ,  \sigma _3   \overline{ \mathbf{{G}}}_{\nu \mu} \rangle     ,
\end{aligned}  \end{equation}
Let us consider the set of  multi--indexes
\begin{equation}\label{multi}
     \mathcal{M}:= \{ \alpha :  \lambda
\cdot  \alpha     > \omega _1 \text{  and } \lambda
\cdot  \alpha -\lambda   _k   < \omega _1   \, \forall \, k \, \text{
s.t. } \alpha _k\neq 0 \} .
\end{equation}
Set also $\Lambda := \{  \lambda  \cdot  \alpha :  \alpha \in \mathcal{M} \}$.

\noindent Like in   \cite{CM1},  there is a  new set of variables
   $\zeta =z+O(z^2)$ s.t.     for a fixed $C$
   \begin{equation}  \label{equation:FGR3} \begin{aligned}   & \| \zeta  -
 z  \| _{L^2_t}
\le CC_0\epsilon ^2\, , \quad  \| \zeta  -
 z \| _{L^\infty _t} \le C \epsilon ^3 \quad \text{ and}
\end{aligned}
\end{equation}
\begin{equation} \label{equation:FGR4} \begin{aligned}
  s_j \im \dot \zeta
 _j&=
\partial _{\overline{\zeta}_j}H_2 '(\zeta , h ) +
\partial _{\overline{\zeta}_j}Z_0 (\zeta , h )+  \mathcal{D}_j
     \\&  +2 \sum  _{ \substack{ \lambda
\cdot  \alpha =\lambda \cdot   \nu
   \\  (\alpha , \nu )\in \mathcal{M}^2 }} \nu _j \frac{\zeta ^{
\alpha } \overline{ \zeta}^ { \nu }}{\overline{\zeta}_j}  \langle  R ^{+}_{\mathcal{K}_{\omega _1
 }  } ( -\lambda  \cdot  \alpha )  \textbf{G}
_{\alpha 0 } ,  \sigma _3   \overline{ \mathbf{{G}}}_{\nu 0} \rangle  ,
\end{aligned}
\end{equation}
  where    for a
fixed constant $c_0$   we have
\begin{equation}\label{eq:FGR7} \sum _{j=1}^{\textbf{n}}\| {\mathcal{D}}_j  \overline{\zeta} _j\|_{
L^1[0,T]}\le c_0 (1+C_0)   \epsilon ^{2}
 . \end{equation}
 Now we consider, like in \cite{CM1},
\begin{align} \label{eq:FGR5}
 &\partial _t \sum _{j=1}^{\textbf{n}} s_j \lambda  _j
 | \zeta _j|^2  =  2  \sum _{j=1}^{\textbf{n}}  \lambda _j\Im \left (
\mathcal{D}_j '\overline{\zeta} _j \right ) -\\&    -4 \sum  _{ \substack{ \lambda
\cdot  \alpha =\lambda \cdot   \nu
   \\  (\alpha , \nu )\in \mathcal{M}^2 }} \lambda  \cdot \nu
\Im  \left ( \zeta ^{ \alpha } \overline{\zeta }^ { \nu  } \langle
R ^{+}_{\mathcal{K}_{\omega _1
 }  } ( -\lambda  \cdot  \alpha )     \textbf{G}_{ \alpha 0} , \sigma _3   \overline{\textbf{ G} }  _{ \nu 0
}\rangle \right ) .\nonumber
\end{align}
In the second line of  \eqref{eq:FGR5} we have a sum
\begin{equation}   \label{eq:FGR8} \begin{aligned} &      \Gamma (\zeta ):=   -   4\sum _{L\in \Lambda     } L     \Im \left    \langle R_{ \mathcal{K}_{\omega _1}}^+ (-L  )
\textbf{G} (L , \zeta),
 \sigma _3 \overline{\textbf{G} (L , \zeta)        }\right \rangle ,   \text{ for } \\& \textbf{G} (L , \zeta)
:=\sum _{
\substack{ \lambda \cdot  \alpha  =L
   \\  \alpha \in \mathcal{M} }}\zeta ^{ \alpha }   \textbf{G}_{ \alpha 0} .
\end{aligned}
\end{equation}
For
   $W  =\lim_{t\to\infty}e^{-\im t \mathcal{K} _{\omega _1} }e^{\im t\sigma_3
(-\Delta +\omega _1 )}$, there exist  $ F ^{(L,\zeta)}  \in W^{k,p}(\R ^2, \C ^2)$
 for all $k\in \R$ and $p \in (1, \infty )$ with
 $ 2\textbf{G} (L , \zeta) =W F^{(L,\zeta)}  $, see   \cite{CT}. Then for
 $^t{ {F}}^{(L,\zeta)}=( {F}_1^{(L,\zeta)}, {F}_2^{(L,\zeta)})  $
 \begin{align}
 \Gamma (\zeta ) &= -4\sum _{L\in \Lambda     }\lim _{\varepsilon \searrow 0} \Im [   \langle R_{  -\Delta    }  (-L - \omega _1 +\im \varepsilon )
{F}_1^{(L,\zeta)},
   \overline{F}_1^{(L,\zeta)}          \rangle - \langle R_{   \Delta   }  ( \omega _1 -L  +\im \varepsilon )
{F}_2^{(L,\zeta)},
   \overline{F}_2^{(L,\zeta)}    \rangle   ]     \nonumber
\\&    =  4\sum _{L\in \Lambda     }\lim _{\varepsilon \searrow 0} \Im   \langle R_{   \Delta   }  ( \omega _1 -L  +\im \varepsilon )
{F}_2^{(L,\zeta)},
   \overline{F}_2^{(L,\zeta)}    \rangle  \label{eq:FGR81}\\&=
-4\sum _{L\in \Lambda     }\lim _{\varepsilon \searrow 0}
\int _{\R ^2} \frac{\varepsilon}{(x ^2- (\Lambda-\omega _1))^2+\varepsilon ^2  }  |\widehat{F}_2^{(L,\zeta)}(x )|^2 dx \le 0.\nonumber
\end{align}
Now we   assume:

\begin{itemize}
\item[(H13)]   for some fixed constant $\Gamma >0$  and for all $\zeta \in
\mathbb{C} ^{\mathbf{n}}$ we have \end{itemize}
\begin{equation} \label{eq:FGR} \begin{aligned} &  \Gamma (\zeta )
 <- \Gamma  \sum _{ \substack{ \alpha \in \mathcal{M}}}  | \zeta ^\alpha  | ^2 .
\end{aligned}
\end{equation}
Then integrating    and exploiting \eqref{equation:FGR3} we get for $t\in [0,T]$
\begin{equation}\label{eq:FGRfgr}  \Gamma \sum _{ \substack{ \text{$\alpha$ as in (H14)}}}  \|z ^\alpha \| _{L^2(0,t)}^2\le
c   C_0  \epsilon ^2 -\sum _j s_j \lambda _j  |z
_j(0)|^2 + \sum _j s_j \lambda _j  |z
_j(t)|^2 .
\end{equation}
We want to conclude  for some other fixed $c'$
\begin{equation}\label{eq:FGRfg1r}  \text{ l.h.s. of \eqref{eq:FGRfgr}}\le
3c '   C_0  \epsilon ^2 .
\end{equation}
 Since $  \sup _{0\le t \le T}\inf _{\vartheta \in \R }  \| u (t) -e^{\im \vartheta}\phi _{\omega_{1} }\| _{H^1}<\epsilon $ by hypothesis,  we can conclude that  $\sum _j   \lambda _j  |z
_j(t)|^2\le c' \epsilon ^2$ for any $t$. Then  \eqref{eq:FGRfgr} implies \eqref{eq:FGRfg1r}
  since here we can assume $C_0>1$ we get \eqref{eq:FGRfg1r}.

We conclude that for $\epsilon _0>0$  sufficiently small and any $T>0$,
\eqref{Strichartzradiation}--\eqref{L^2discrete} in  $I=[0,T]$ and with $C=C_0$  implies
\eqref{Strichartzradiation}--\eqref{L^2discrete} in  $I=[0,T]$  with $C=c(1+\sqrt{C_0}  )$ for $c$.
This yields Proposition  \ref{prop:mainbounds}.

 \qed

\begin{remark}
Notice that in the proof of asymptotic stability in \cite{Cu2}, where  $s_j=-1$  for all $j$, orbital stability is a consequence of  \eqref{eq:FGRfg1r} rather than the other way around. Indeed in that case we have
\begin{equation}\label{eqrem}  \sum _j   \lambda _j  |z
_j(t)|^2 +   \Gamma \sum _{ \substack{ \text{$\alpha$ as in (H14)}}}  \|z ^\alpha \| _{L^2(0,t)}^2\le
c   C_0  \epsilon ^2 +\sum _j   \lambda _j  |z
_j(0)|^2,
\end{equation}
and taking the initial datum $u_0$ very close to $ \phi _{\omega_{1} }$ we can assume   $\sum _j  \lambda _j  |z
_j(0)|^2 \le c' \epsilon ^2 $. Then each term in the l.h.s. of \eqref{eqrem}  is small for all $t>0$. Furthermore this and and the fact that from \eqref{eq:eqz1} we derive that the
time derivatives $\dot z_j (t)$ remain small, we conclude that  $z_j(t) \stackrel{t\to \infty}{\rightarrow} 0$  for all $j$.
\end{remark}

\section{   Theorem \ref{theorem-1.1} and cubic quintic equations}
\label{sec:applications}

We have not carried out numerical experiments to check examples
to which Theorem \ref{theorem-1.1} applies.   In this section we
combine numerical results in  \cite{pego} with  number of
assumptions to propose some possible applications of Theorem \ref{theorem-1.1}. We   discuss mainly  hypothesis (H14).

We consider the cubic--quintic NLS \eqref{NLS3-5}.
For each $m\geq 0$ one can find a family of vortices $e^{\im m \theta}\psi_{\omega}(r)$ for $\omega\in \mathcal O=(0,\omega_*)$ with $\psi_\omega\geq 0$ \cite{IaiaWarchall99}.
Here, the upper bound $\omega_*$ of $\mathcal O$ is given by $\omega_*:=\sup\{\omega>0\ |\ \exists s>0,\ \frac{\omega}{2}s^2-\frac{1}{4}s^4+\frac{1}{6}s^6<0\}$ which is in this case $\frac{3}{16}$.
Notice that if $\frac{\omega}{2}s^2-\frac{1}{4}s^4+\frac{1}{6}s^6\geq 0$ for all $s>0$, by Pohozaev identity we can show that there are no nontrivial bound states in the energy space.

As we have already discussed in Section \ref{section:introduction},
for each $m=1,2,3,4,5$ \cite{pego} shows that there is a critical value $ \omega_{cr}$ such that for $\omega \ge \omega_{cr}$ the vortices are spectrally
stable, while for $\omega < \omega_{cr}$ they are spectrally
unstable. For $m=1 $  the value is
$\omega_{cr}\approx 0.1487$, see Section 5.5  \cite{pego}. In all these cases  spectral instability is
generated as follows.
   As $\omega  $ approaches $  \omega_{cr}^+$  from above,    two distinct   eigenvalues on the imaginary axis
$\im \lambda ^{(1)}(\omega )$  and $\im \lambda ^{(2)}(\omega )$ coalesce  at $\omega = \omega_{cr}$ at a point $ \im \lambda   _{cr} $  ($ \im \lambda   _{cr} \approx \im 0.0478$  for $m=1$). As $\omega$ decreases further, two eigenvalues bifurcate from  $ \im \lambda   _{cr}  $  out of
the imaginary axis. In \cite{pego}  it is not stated explicitly whether or not
the eigenvalues $ \im \lambda ^{(j)}(\omega ) $  for $j=1,2$ are simple and whether their algebraic and geometric dimensions
    coincide.  The fact that only   eigenvalues with different signatures can generate
    by collapse  eigenvalues  outside $\im \R$ is well known, and we formalize it
    as follows.

In the following, $\mathcal L_\omega$ will denote the linearized operator of the vortices $e^{\im \omega t} e^{\im m \theta}\psi_\omega(r)$ of the cubic-quintic NLS \eqref{NLS3-5}.

   \begin{lemma} \label{lem:bif} Consider equation \eqref{NLS3-5},
   an $ m=1,2,3,4,5$ , a corresponding vortex and the operators $\mathcal{L}_{\omega}$.
   Suppose that there exists $\varepsilon _0>0$ s.t.
   for
   $\omega \in (\omega _{cr}, \omega _{cr} +\varepsilon _0)$
   the eigenvalues $ \im \lambda ^{(j)}(\omega ) $  for $j=1,2$ are simple and their algebraic and geometric dimensions
    coincide (and so are 1). Then one of them satisfies (H14).
\end{lemma}
\proof  Recall, preliminarily,  that if $z$ is an   eigenvalue of $\mathcal{L}_{\omega}$
with $ z\neq -\overline{z}$, that is if $z\not \in \im \R$, then $\Omega (v,\overline{v})=0$
for any $v\in \ker (\mathcal{L}_{\omega} -z)$. This follows from
\begin{equation} \label{eq:bif1}
    z \langle J ^{-1}v,\overline{v}\rangle = \langle J ^{-1}\mathcal{L}_{\omega} v ,\overline{v}\rangle =- \langle v ,\overline{J ^{-1}\mathcal{L}_{\omega}v}\rangle  =- \overline{z} \langle J ^{-1}v,\overline{v}\rangle ,
\end{equation}
which uses $J ^{-1}\mathcal{L}_{\omega}=- \mathcal{L}_{\omega} ^* J ^{-1}$.
If  $z=\im \lambda \in \R $ we have
\begin{equation*}
    \im \lambda  \langle J ^{-1}v,\overline{v}\rangle =  \overline{\im \lambda  \langle J ^{-1}v,\overline{v}\rangle} ,
\end{equation*}
so that $\langle J ^{-1}\mathcal{L}_{\omega} v ,\overline{v}\rangle \in \R$
for any $v\in \ker (\mathcal{L}_{\omega} - \im \lambda )$. When $\im \lambda \in \R $
is simple we have $\langle J ^{-1}\mathcal{L}_{\omega} v ,\overline{v}\rangle \neq 0$
for any $v\in \ker (\mathcal{L}_{\omega} - \im \lambda )\backslash \{ 0\} $ and the sign is
called the \textit{Krein signature}. Since
\begin{equation*}   \langle J ^{-1}\mathcal{L}_{\omega} v ,\overline{v}\rangle
    = \im \lambda  \langle J ^{-1}v,\overline{v}\rangle,
\end{equation*}
it is clear that the Krein  signature of $\im \lambda$ is positive exactly if  $\langle J ^{-1}v,\overline{v}\rangle = \im \varsigma$  with $\varsigma <0$
(this explains why in (H14) the hypothesis $s_j=1$ implies \textit{negative} Krein signature and that positive Krein signature corresponds to $s_j=-1$). Notice that for $\im \lambda$
to have a well defined Krein signature it is not necessary that be simple, and it
is sufficient   that  $\langle J ^{-1}\mathcal{L}_{\omega} v ,\overline{v}\rangle  $
have constant  sign as $v\neq 0$ varies in $  \ker (\mathcal{L}_{\omega} - \im \lambda )$.

Now let us proceed with the proof  of Lemma \ref{lem:bif}. By hypothesis the  $\im \lambda ^{(j)}(\omega )$ for $\omega >\omega_{cr}$ are simple and hence they have a well defined Krein signature, which is constant in
$\omega >\omega_{cr}$  (see \cite{kollar1}) and can be written in the form  $ s ^{(j)}:=-\im \Omega (\xi _{\omega}^{(j)} ,\overline{\xi} _{\omega}^{(j)})$
for appropriate generators $\xi _{\omega}^{(j)} \in \ker ( \mathcal{L}_\omega - \im \lambda ^{(j)}(\omega )) $. We have $ s ^{(j)}  \in \{ -1, 1\}$. For $\omega <\omega_{cr}$ in \cite{pego} it is proved that
  there are two eigenvalues $ z ^{(1)}(\omega ) $ and    $ z ^{(1)}(\omega ) $ with $\Re   z ^{(j)}(\omega )\neq 0$, which  for reasons of symmetry  satisfy $ z ^{(2)}(\omega )   = - \overline{z} ^{(1)}(\omega )  $.  These two  eigenvalues exit from $ \im \lambda ^{(j)}(\omega _{cr}) $  as $\omega$ decreases.
With an argument by contradiction we suppose now that
  $ s ^{(1)}= s ^{(2)}$ for $\omega > \omega_{cr}$. Then, by   Section 6.1 in \cite{kollar1},  this continues to be true for $\omega < \omega_{cr}$ in the sense that
  the quadratic form $\langle J ^{-1}\mathcal{L}_{\omega} \cdot ,\overline{\cdot }\rangle  $,
  which is definite in $\ker ( \mathcal{L}_\omega - \im \lambda ^{(1)}(\omega ))\oplus  \ker ( \mathcal{L}_\omega - \im \lambda ^{(2)}(\omega )) $ for $\omega > \omega_{cr}$ must continue to be definite also in  $\ker ( \mathcal{L}_\omega - z ^{(1)}(\omega ))\oplus  \ker ( \mathcal{L}_\omega - z ^{(2)}(\omega )) $ for $\omega < \omega_{cr}$. But from what we saw
in \eqref{eq:bif1} this is not possible.  This gets us to a contradiction. So  $ s ^{(1)}\neq s ^{(2)}$ for $\omega >\omega_{cr}$ and one of the two must be equal to $1$.
\qed

\bigskip

Since \eqref{NLS3-5}  is translation invariant    it is beyond the
scope of our theory.
In order to find examples of equations  not translation invariant   which satisfy (H14)
it is natural  to add
to  \eqref{NLS3-5} a small
potential $\varepsilon V(|x|) $ with a point of relative minimum in 0.
We will show that  the perturbed equation has vortices. As $\varepsilon \to 0$ they converge   to vortices of  \eqref{NLS3-5}
in any  space $\Sigma _k(\R ^2, \C)$.

Then the spectrum and the eigenfunctions of the linearizations
 $\mathcal{L}^{(\varepsilon)}_{\omega}$  converge  to spectrum and  eigenfunctions of $\mathcal{L} _{\omega}$. In particular,
  assuming our mixed rigorous and  numerical proof that \eqref{NLS3-5}
  satisfies (H14) for $\omega >\omega _{cr}$, then we  will have  obtained
  this result  also for the operators $\mathcal{L}^{\varepsilon}_{\omega}$
  with $\varepsilon \neq 0$.

  Let $\phi_\omega (e^{\im  \theta}  r) = e^{\im m \theta} \psi_\omega (r)$ be the vortex of \cite{pego} with $\psi_\omega\geq 0$.
Under the assumption that   for  \eqref{NLS3-5}
the kernel of $\mathcal L_\omega$ restricted in $L^2_m:=\{u\in L^2\ |\ e^{-\im m \theta}u\ \mathrm{is\ radially\ symmetric}\}$ is $\mathrm{Span}\{J\phi_\omega\}$ (in \cite{pego} it is shown that numerically this appears to be  generically true)
we show that for small $\varepsilon $ there exists $\phi_{\omega,\varepsilon}(e^{\im  \theta}  r)=e^{\im m \theta }\psi_{\omega,\varepsilon}(r)$ which is a solution of
\begin{align}\label{eq:perturb1}
0=-\Delta \phi_{\omega,\varepsilon} + \omega \phi_{\omega,\varepsilon} + \varepsilon V \phi_{\omega,\varepsilon} -|\phi_{\omega,\varepsilon}|^2\phi_{\omega,\varepsilon}+|\phi_{\omega,\varepsilon}|^4\phi_{\omega,\varepsilon}.
\end{align}

More precisely, we have the following proposition.
\begin{proposition}\label{prop:perturb2}
Assume $\mathrm{ker}\left.\mathcal L _{\omega _0}\right|_{L^2_m}=\{J\phi_{\omega _0}\}$.
Then  there exist  $\delta _0>0$ and
 $\varepsilon _0>0$  s.t. for $ \varepsilon   \in (-\varepsilon _0,\varepsilon _0 )$
 and $\omega \in (\omega _0-\delta _0, \omega _0+\delta _0)$
 there exists $\phi_{\omega,\varepsilon}\in \cap_{k\geq 0}\Sigma _k(\R^2,\C)$ which satisfies \eqref{eq:perturb1}.
Furthermore,   the map $(\omega , \varepsilon )\to \phi_{\omega,\varepsilon} $ is
 in $C ^1((\omega _0-\delta _0, \omega _0+\delta _0) \times (-\varepsilon _0,\varepsilon _0) , \Sigma_k)$ for arbitrary $k\geq 0$.
\end{proposition}

To prove Proposition \ref{prop:perturb2}  we consider a preparatory and  standard lemma.
\begin{lemma}\label{lem:perturb3}
Assume $\mathrm{ker}\left.\mathcal L _{\omega  }\right|_{L^2_m}=\{J\phi_{\omega  }\}$.  Then $A_\omega = -\Delta_r + \frac{m^2}{r^2}+ \omega -3\psi_\omega^2+5\psi_\omega ^4$ is invertible in $L^2_{rad}(\R^2,\R)$ and $\|e^{\im m \theta}A_\omega^{-1}u\|_{H^2}\lesssim \|u\|_{L^2}$.
\end{lemma}

\begin{proof}
Let $v\in L^2_{rad}(\R^2,\R)$ satisfy $A_\omega v=0$.
Then, multiplying by $e^{\im m \theta}$  we have
\begin{align*}
(-\Delta + \omega -2|\phi_\omega|^2+3|\phi_\omega|^4)u+\(-\phi_\omega^2+ 2\phi_\omega^2|\phi_\omega|^2\)\bar u=0,
\end{align*}
where $u=e^{\im m \theta}v$.
Therefore, using the natural identification between $\C$ and $\R^2$, we have
\begin{align*}
\mathcal L_\omega u = 0.
\end{align*}
By the assumption, we have $u= a J \phi_\omega$ and thus $v=a J\psi _\omega$.
However, $v$  has values in $\R \times \{0\}$ while $ J\psi _\omega$ has values in $\{0\} \times \R$.
So $a=0$ and   $\mathrm{ker}A_\omega=\{0\}$.
Therefore, $A_\omega$ is invertible.
Finally suppose $A_\omega v =u$.
Then, first we have $\|e^{\im m \theta}v\|_{L^2}=\|v\|_{L^2}\lesssim \|u\|_{L^2}$.
Next, multiplying by $e^{\im m \theta}$, we have
\begin{align*}
(-\Delta + \omega )e^{\im m \theta}v = \(3\psi_\omega^2-5\psi_\omega^4\) e^{\im m \theta}v + e^{\im m \theta} u.
\end{align*}
Taking the $L^2$ norm of both sides, we have the conclusion.
\end{proof}

\begin{proof}[Proof of Proposition \ref{prop:perturb2}]
Set $\delta =\omega -\omega _0$ and consider
$\phi_{\omega _0 + \delta ,\varepsilon}=e^{\im m \theta}(\psi _{\omega _0 + \delta} + v _{\delta ,\varepsilon})$, where $v _{\delta ,\varepsilon}$ is radially symmetric and real valued.
Then, substituting this into \eqref{eq:perturb1}  and for  $v=v _{\delta ,\varepsilon} $, we have
\begin{align}\label{eq:perturb4}
A_{\omega _0  } v =   G   ( \delta , \varepsilon , v) \text{ where } G   ( \delta , \varepsilon , v):= -\varepsilon V (\psi_{\omega _0 + \delta} + v ) + \left [\left ( 3  \psi_{\tau }  ^2   -  5\psi_\tau ^4 \right ) \right ] _{\omega _0}^{\omega _0 + \delta}v -    \delta v   +N_{\omega _0 + \delta}(v  ),
\end{align}
with $N_{\omega _0 + \delta}(v  )$     nonlinear in $v $   and   the convention $\left[ f(\tau ) \right ] _{a}^{b}   =f(b)-f(a)$.
We can rephrase \eqref{eq:perturb4} as   \begin{equation}  \label{eq:perturb41}  \begin{aligned}   & F(\delta ,\varepsilon , v) =0 \text{ where }  F(\delta ,\varepsilon , v):= v-A _{\omega _0}^{-1} G   ( \delta , \varepsilon , v)  .
\end{aligned}
\end{equation}
The function $F$ is in $ C^1( (\omega _0-\delta _0, \omega _0+\delta _0) \times (-\varepsilon _0,\varepsilon _0) \times H^2, H^2  )$. An elementary application of the implicit function theorem yields a function $v _{\delta , \varepsilon}  $ in $ C^1( (\omega _0-\delta _0, \omega _0+\delta _0) \times (-\varepsilon _0,\varepsilon _0)  , H^2  )$.
By a standard bootstrapping argument,   $H^2$ can be replaced by $\Sigma _k$ for arbitrary $k$.
\end{proof}

 In \cite{pego}  it is checked numerically that for a generic vortex of the \eqref{NLS3-5}  \begin{equation} \label{eq:instpa1}\begin{aligned}
&  N_g ( {\mathcal L}_\omega ) = \text{Span}\{ J \phi _\omega ,  \partial _{x_1}\phi _\omega ,  \partial _{x_2}\phi _\omega , \partial _{ \omega}    \phi _\omega  , J{x_1}\phi _\omega   , J{x_2}\phi _\omega \} .
\end{aligned}
\end{equation}
  We have chosen $V(|x|)$  with a relative  minimum at 0. The following fact is
well known, see Theorem 4.1 \cite{zhousigal} and Theorem 3.0.2 \cite{zhousigal1}.

\begin{lemma} \label{lem:zhou}  Consider an $m=1$ vortex  of the\eqref{NLS3-5}.
Assume $\frac{d}{d\omega} \|\phi _\omega \|_2\neq 0$ and  \eqref{eq:instpa1}   for the linearized operator $\mathcal{L}_\omega$. Then for $\varepsilon > 0$  sufficiently small we have \begin{equation} \label{eq:zhou1}\begin{aligned}
&\ker {\mathcal L}_\omega ^{(\varepsilon)} =\text{Span}\{ J     \phi  _{\omega ,\varepsilon }    \} \text{ and } N_g ( {\mathcal L}_\omega  ^{(\varepsilon)}) = \text{Span}\{ J \phi _{\omega ,\varepsilon }  , \partial _{ \omega}    \phi _{\omega ,\varepsilon }   \} .
\end{aligned}
\end{equation}
Furthermore, ${\mathcal L}_\omega ^{(\varepsilon)} $ has an eigenvalue of algebraic and geometric multiplicity 2  which is  of the form
$\im \mu (\varepsilon)=\im\varepsilon \sqrt{2e } +o(\varepsilon) $,  with $e $   the eigenvalue  of the
Hessian matrix of the potential in 0, and  with  eigenfunctions
\begin{equation}\label{eq:zhou2}
    \Psi  _j ^{(\varepsilon)}=      \sqrt{2} \partial _{x_j}  \phi _{\omega   }  +\im \sqrt{e}\varepsilon J{x_j}  \phi _{\omega   }  +o(\epsilon),
\end{equation}
with $\|  o  (\varepsilon ) \| _{\Sigma _k}=o (\varepsilon ) $  for any $k$.
\end{lemma}
\proof Everything is proved in Theorem 3.0.2 \cite{zhousigal1}. We only
remark that there is only one small eigenvalue $\im \mu (\varepsilon)\in \im \R_+$  which has to be of multiplicity 2.
Indeed,  ${\mathcal L}_\omega  ^{(\varepsilon)}  \Psi  _1 ^{(\varepsilon)} = \im \mu (\varepsilon)\Psi  _1 ^{(\varepsilon)}$ implies by the symmetry
  $J ^{-1}\mathcal{V}_\omega (x_1, x_2 )  J=\mathcal{V}_\omega ( x_2, -x_1 ) $
  also   ${\mathcal L}_\omega  ^{(\varepsilon)}  J\Psi  _1 ^{(\varepsilon)} ( x_2, -x_1 ) = \im \mu (\varepsilon)J\Psi  _1 ^{(\varepsilon)} ( x_2, -x_1 )$.  By     $J \phi _{\omega   } (x_1, x_2) =-   \phi _{\omega   } ( x_2, -x_1 )$
  we have
 \begin{equation*}
\begin{aligned}
&  J\Psi  _1 ^{(\varepsilon)} ( x_2, -x_1 )=  J \sqrt{2} (\partial _{x_1}  \phi _{\omega   }) ( x_2, -x_1 ) -\im \sqrt{e}\varepsilon  {x_2}  \phi _{\omega   }( x_2, -x_1 )+o(\epsilon) =\\&      \sqrt{2}  \partial _{x_2} [ J\phi _{\omega   }  ( x_2, -x_1 ) ]+\im \sqrt{e}\varepsilon J{x_2}  \phi _{\omega   }( x_1, x_2 )+o(\epsilon) =  \sqrt{2}  \partial _{x_2}  \phi _{\omega }  ( x_1, x_2 )  +\im \sqrt{e}\varepsilon J{x_2}  \phi _{\omega   }( x_1, x_2 )+o(\epsilon)
\end{aligned}
\end{equation*}
   and hence necessarily
   $J\Psi  _1 ^{(\varepsilon)} ( x_2, -x_1 )=\Psi  _2 ^{(\varepsilon)} ( x_1, x_2) $. This implies that $\im \mu (\varepsilon)$ has multiplicity 2.

\qed

In the following remarks we discuss whether the perturbations of equation  \eqref{NLS3-5} satisfy the hypotheses
of Theorem \ref{theorem-1.1}.

\begin{remark}
\label{rem:pego1-3}
  (H1)--(H3) are trivially satisfied.\end{remark}
\begin{remark}
\label{rem:pego4}
 It is known that   equation  \eqref{NLS3-5} satisfies (H4) and the same holds for the
perturbations.\end{remark}
 \begin{remark}
\label{rem:pego5}
According to
the numerical experiments in \cite{pego}  hypothesis  (H5)
is true generically.  The same will be true for the perturbations.\end{remark}
\begin{remark}
\label{rem:pego6}
 (H6)  is proved numerically for $\omega \ge \omega _{cr}$ in \cite{pego} and for the perturbations is a
consequence of   Proposition \ref{prop:perturb2}  and Lemma    \ref{lem:zhou}.\end{remark}
 \begin{remark}
\label{rem:pego7}
 Assuming  the numerical results in
\cite{pego} which claim that \eqref{eq:instpa1} is true for generic $\omega$, then for the perturbations  (H7)
is a consequence of Lemma \ref{lem:zhou}.\end{remark}

 \begin{remark}
\label{rem:pego8-12} We don't know the status of hypotheses
(H8)--(H12) for \eqref{NLS3-5}, but if they hold, they hold also for the
perturbations.

\noindent As we have already mentioned,
failure of (H8) would yield some Jordan block of dimension 2 or higher forcing $e^{t \mathcal{L}_\omega}$ to grow algebraically
in the invariant space $N_g^\perp
(\mathcal{L}_\omega ^{\ast}) $   in  \eqref{eq:begspectdec2}  producing essentially a   linear instability, which would yield
an easy to detect  nonlinear instability. Since this is not what the numerical experiments show, we  conclude that probably
\cite{pego,Towers} confirm (H8).

\noindent In \cite{pego} embedded eigenvalues
and singularities at the edges are not discussed explicitly but hypotheses
(H9)--(H10)  seem to be confirmed. In \cite{pego} isolated eigenvalues are obtained as zeros of an Evans function.
It is observed, see the discussion on pp. 371--372, that sometimes the Evans function is small near the continuous spectrum.
This smallness is attributed not to eigenvalues sitting in the continuous spectrum, but rather  to \textit{resonances} on the other side
of the continuous spectrum (in essence, to zeros of an analytic continuation of the Evans function beyond the continuous spectrum).

\noindent Hypotheses (H11)-(H12) state  that the eigenvalues  between 0 and $\im \omega$
are positioned in a generic way. This is plausible to expect and probably can be proved for  $V$ generic. We did not attempt the proof.
  \end{remark}

\begin{remark}
\label{rem:pego13}
 We have discussed at length Hypothesis (H13) in Section \ref{section:introduction}. It ought to be checked
 directly for perturbations  of \eqref{NLS3-5} or  for equation \eqref{NLS3-5}. Notice
 that, since \eqref{NLS3-5} is translation invariant,  the search of an effective Hamiltonian is somewhat more involved, see \cite{Cu0,CM1}.\end{remark}

\begin{remark}
\label{rem:pego14}
 We have already discussed (H14) which is true, assuming that each of the    eigenvalues $ \im \lambda ^{(j)}(\omega ) $    has some fixed Krein signature.
Indeed this  is what happens if  each eigenvalue  is simple, as we assumed in {Lemma} \ref{lem:bif}. Even if there is an eigenvalue of higher multiplicity,  it is enough to ask for  $\im \Omega (v,\overline{v})$ to have fixed sign for
any eigenvector. This property on the signatures continue to hold for the perturbations.\end{remark}

 \noindent

\appendix

\section{Appendix: proof of {Lemma} \ref{lem:conditional4.2} }
 \label{app:4.2}

First of all, it is equivalent to consider equation \ref{eq:eqh}  for $h$.
Let $ X_c  = M ^{-1}L^2_c(\omega _1)$ and by an abuse of notation
let us set $\widetilde{P}_c = M ^{-1}P_c(\omega _1) M$, where $P_c$ is introduced under Lemma \ref{lem:Specdec}. Set also $\mathcal{K}= \mathcal{K}_{\omega _1}$
 The following  three lemmas are  Lemma 3.1--3.3 in \cite{CT}.
\begin{lemma}[Strichartz estimate]
  \label{lem:lem3.1}
 There exists a positive
number $C $   such that
for any $k\in [0,2]$:
  {\item {(a)}}
 for any $h= \widetilde{P}_c h$ and any admissible  all pair  $(p,q)$,
$$\|e^{-\im  t\mathcal{K} } h\|_{L_t^pW_x^{k,q } }\le C\|h\|_{H^k};$$
 {\item {(b)}}
  for any $g(t,x)\in
S(\R^2)$ and any couple of admissible pairs $(p_{1},q_{1})$
$(p_{2},q_{2})$ we have
$$
\|\int_{0}^te^{-\im (t-s)\mathcal{K} } \widetilde{P}_c g(s,\cdot)ds\|_{L_t^{p_1}W_x^{k,q_1 } } \le
C\|g\|_{L_t^{p_2'}W_x^{k,q_2' } }.
$$\end{lemma}

\begin{lemma}
  \label{lem:lem3.2} Let $s>1$. $\exists$   $C=C $  such that:
   {\item {(a)}}
  for any $f\in S(\R ^2 )$,
$$  \|  e^{-\im t \mathcal{K}}\widetilde{P}_c f\|_{  L^2_tL_x^{2,-s}} \le
 C\|f\|_{L^2} ;
  $$
  {\item {(b)}}
  for any $g(t,x)\in
 {S}(\R^2)$
$$ \left\|\int  _{\R} e^{\im t\mathcal{K}}\widetilde{P}_c g(t,\cdot)dt\right\|_{L^2_x} \le
C\| g\|_{ L_t^2L_x^{2,s}}.
$$\end{lemma}

\begin{lemma}
  \label{lem:lem3.3} Let $s>1$.  $\exists$   $C $
such that $\forall$ $g(t,x)\in {S}(\R^2)$ and $t\in\R$:
$$   \left\|  \int_0^t e^{-\im (t-s)\mathcal{K}}\widetilde{P}_c g(s,\cdot)ds\right\|_{  L^2_tL_x^{2,-s}} \le C\|
g\|_{  L^2_tL_x^{2, s}} .
$$
\end{lemma}

\begin{lemma}
  \label{lem:lem3.4} Let $(p,q)$ be an admissible pair and let
$s>1$. $\exists$  a constant  $C>0$  such that $\forall$
$g(t,x)\in {S}(\R^2)$ and $t\in\R$:
$$
\left\|\int_0^t e^{-\im (t-s)\mathcal{K}}\widetilde{P}_cg(s,\cdot)ds
\right\|_{L_t^pL_x^q} \le C\|g\|_{  L^2_tL_x^{2, s}}.$$
\end{lemma}

The following is Proposition 1.2 in \cite{CT}.
\begin{lemma}
  \label{lem:prop1.2}
The following limits are well defined isomorphism, inverse of each
other:
 \begin{equation*}
\begin{aligned}
  &   W u= \lim _{t\to +\infty } e^{ \im  t\mathcal{K} } e^{  \im t \sigma _3(\Delta -\omega _1)
 }u  \text{ for any $u\in L^2$}  \\&
Z u= \lim _{t\to +\infty }  e^{ \im t(-\Delta +\omega _1) } e^{ -\im t \mathcal{K}
 }  \text{ for any $u=\widetilde{P}_cu$.} \end{aligned}
\end{equation*}
 For any $p\in (1,\infty )$  and any $k$ the restrictions of
$W$ and $Z$ to $L^2\cap W^{k,p}$ extend into operators such that for
for a constant $C $  we have $$\| W\| _{
W^{k,p}(\R ^2),   W^{k,p}_c}+\| Z\| _{
  W^{k,p}_c, W^{k,p}(\R ^2) }<C  $$ with
$W^{k,p}  _c  $ the closure in $W^{k,p} (\R ^2)$ of $
W^{k,p}(\R ^2)\cap \widetilde{P}_c L^2  _c  $.
\end{lemma}
The following is Lemma 3.5 \cite{CT}.

\begin{lemma}
  \label{lem:lem3.5} Consider the diagonal matrices $
E_+=\text{diag}(1 , 0)$ $E_-=\text{diag}(0 , 1).$ Set $P_\pm  =Z  E_\pm W $ with $Z $ and $W $
the wave operators associated to $ \mathcal{K} $. Then  we have for
$u =\widetilde{P}_cu$
\begin{equation*}
\begin{aligned}
  &    P_+ u =\lim _{\epsilon \to 0^+}
 \frac 1{2\pi \im }
\lim _{M \to +\infty} \int _\omega ^M \left [ R_{ \mathcal{K} }(\lambda
+\im  \epsilon )- R_{\mathcal{K} }(\lambda -\im \epsilon )
 \right ] ud\lambda \\&
P_- u =\lim _{\epsilon \to 0^+}
 \frac 1{2\pi \im }
\lim _{M \to +\infty} \int _{-M }^{-\omega } \left [ R_{\mathcal{K}
}(\lambda +\im \epsilon )- R_{\mathcal{K} }(\lambda -\im \epsilon ) \right ]
ud\lambda \end{aligned}
\end{equation*}
and  for any $s_1$ and $s_2$ and for $C=C (s_1,s_2   )$   we have $$  \| (P_+ -P_- -\widetilde{P}_c \sigma _3) f\|  _{L^{2,s_1} }\le
C  \|       f\|  _{L^{2,s_2} }.  $$
\end{lemma}

Now we look at the term $\mathbf{E}$  in \eqref{eq:eqh}.
\begin{lemma}\label{lem:rem1}For any preassigned $s$ and for $\epsilon _0 >0$  small enough we have
   \begin{equation} \label{eq:nablaH3} \begin{aligned} &
    \mathbf{E} =    R_1+R_2  \text{ with  }
  \| R_1 \| _{L^1_t([0,T],H^{ 1 }_x)} +\|
   R_2 \| _{L^{2
 }_t([0,T],H^{ 1
,
 s}_x)}\le C( s,C_0
 ) \epsilon^2 .   \end{aligned}  \end{equation}
Furthermore  for a fixed constant $c $ we have \begin{equation} \label{eq:nablaH2} \begin{aligned} &
     \| A \| _{L^\infty( (0,T), \R )}\le c C_0^2 \epsilon ^2
		  .\end{aligned}  \end{equation}

\end{lemma}
\proof The estimate on $A=A'+{A} ^{\prime\prime}  $ follows from the definitions of $A'$ in \eqref{eq:nablaH1} and of $A^{\prime\prime}  $  in \eqref{eq:nablaH12}.

\noindent $\mathbf{E}$  is a sum of various terms. For example we have
\begin{equation*} \begin{aligned}
  &
\|   z^\mu \overline{z}^\nu   M^{-1} [{G}_{\mu \nu}  (Q,0)  - {G}_{\mu \nu}  (Q,Q(f)) ] \| _{L^{2
 }_t([0,T],H^{ 1
,
 s}_x)} \\& \le \|   z^\mu \overline{z}^\nu   \| _{L^{2
 }_t [0,T] }   \|   {G}_{\mu \nu}  (Q,0)  - {G}_{\mu \nu}  (Q,Q(f))  \| _{L^{\infty
 }_t([0,T],H^{ 1
,
 s}_x)} \lesssim C_0^3 \epsilon ^3. \end{aligned}
\end{equation*}
So this term can be absorbed in $R_2$.   Another example is   $\beta (|f|^2)f =\chi _{|f|\le 1}\beta (|f|^2)f +\chi _{|f|\ge 1}\beta (|f|^2)f $.
The 1st term can be bounded, schematically,  by
\begin{equation}\label{pwcub} \begin{aligned}
  &  \|     |f  |^2  f  \| _{  L^{1
 }_t([0,T],H^{ 1
 }_x) } \lesssim \left
\|  \| f \| _{W^{1,6}_x}    \| f \| ^2_{L^{ 6}_x} \right \| _{L^1_t[0,T]
} \le \| f \| _{L^3_t([0,T],W^{1,6}_x)}^3
\  \lesssim C_0^3 \epsilon ^3  \end{aligned}
\end{equation}
 while the 2nd term can be bounded by
 \begin{equation}\label{pwL} \begin{aligned}
  &   \|     f  ^{{L}}   \| _{  L_t^1H^1_x } \lesssim
\left \|  \| f \| _{W^{1,2{L}}_x}    \| f \| ^{{L}-1}_{L^{
2{L}}_x} \right \| _{L^1_t } \le \| f \| _{L^{\frac{
2{L}}{{L}-1}}_tW^{1,2{L}}_x}
 \| f \| ^{{L}-1}_{L^{  2{L}  \frac{{L}-1}{{L}+1}  }_tW^{1,2L}_x} \lesssim C_0^L\epsilon ^{L},  \end{aligned}
\end{equation}
  where in the last step we use $ \| f \| _{L^{  2{L}
\frac{{L}-1}{{L}+1} }_tW^{1,2{L}}_x}\lesssim \| f \| ^{\alpha
}_{L^{\frac{ 2{L}}{{L}-1}}_tL^{ 2{L}}_x} \|
   f      \| _{  L_{t }^\infty H^1_x } ^{1-\alpha }$ for some
   $0<\alpha <1$ by ${L}>3$ (which we can always assume), interpolation and Sobolev embedding.

\noindent Notice that  by $\nabla _{f} \resto ^{1,2}_{k,m}(Q ,\varrho, f) _{|\varrho =Q (f)}=
S^{1,1}_{k,m-1}(Q ,Q (f), f)$ we have by \eqref{eq:opSymb}
\begin{equation*} \begin{aligned}
  &   \|    \nabla _{f} \resto ^{1,2}_{k,m}(Q ,\varrho, f) _{|\varrho =Q (f)}  \| _{  L_t^2H ^{1,s}_x }  \le    \|     \| f \| _{L^{2,-\sigma}_x}   \| _{  L_t^2  }   (\|  f\|
  _{L^2}+|z | +|Q (f)  |) \| _{  L_t^2   }\\& \le   \| f \| _{L_t^2L^{2,-\sigma}_x}  (\| z \| _{  L_t^\infty}     + \| f \| _{  L_t^\infty}L^2) \le 2 C_0 ^2 \epsilon ^2. \end{aligned}
\end{equation*}
Consider for example the contribution of

\begin{equation}\label{el1} \begin{aligned}
  & \nabla _{f}
\int _{\mathbb{R}^2}
B_L  (x, f(x),  Q , z,\varrho  , f  )  f^{   L}(x) dx _{|\varrho  =Q(f)} \sim B_L  (x, f(x),  Q , z,Q(f)  , f  )  f^{   L-1}(x)\\& + \int _{\mathbb{R}^2}
\partial _{6} B_L  (x, f(x),  Q , z,Q(f)  , f  )   f^{   L}(x) dx +  \partial _{2} B_L  (x, f(x),  Q , z,Q(f)  , f  )   f^{   L}(x) .\end{aligned}
\end{equation}
 The last term can be treated like $ f^{   L}$   above, since $\| B_L  (x, f(x),  Q , z,Q(f)  , f  )\| _{  L _{tx} ^\infty} \le C$   by \eqref{eq:B5}.
 We can use \eqref{pwcub}  or  \eqref{pwL}
 for the 1st term of the r.h.s., since
 $L-1\ge 3$.  Finally let us consider the 2nd term in the r.h.side.  If we take $g\in L ^{\infty}_tH ^{-1}_x$ we need to bound
 \begin{equation*} \begin{aligned}
  & \int _{0}^{T} dt \int _{\mathbb{R}^2} | \langle g,
\partial _{6} B_L  (x , f(x  ),  Q , z,Q(f)  , f  )\rangle _{L_{x'}^2}   f^{   L}(x) dx  | \\& \le   \int _{0}^{T} dt  \| | \langle g,
\partial _{6} B_L  (x , f(x  ),  Q , z,Q(f)  , f  )\rangle _{L_{x'}^2}\| _{L^2_x} \|  f^{   L}\| _{L^2_x}\\& \le \| | \langle g,
\partial _{6} B_L  (x , f(x  ),  Q , z,Q(f)  , f  )\rangle _{L_{x'}^2}\| _{L^\infty _tL^2_x}\|  f^{   L}\| _{L^1 _tL^2_x}
  \end{aligned}
\end{equation*}
and we bound the last factor by    \eqref{pwL}. We have  for fixed $t$
 \begin{equation*} \begin{aligned}
  &   \| | \langle g,
\partial _{6} B_L  (x , f(x  ),  Q , z,Q(f)  , f  )\rangle _{L_{x'}^2}\| _{ L^2_x}\le  \| \partial _{6} B_L \| _{B(\Sigma_{-k},\Sigma _k) } \| g \| _{\Sigma_{-k}}
  \end{aligned}
\end{equation*}
  so that by \eqref{eq:B5}, or by its analogue for the $B_L$ in Lemma
  \ref{lem:ExpH11}, we have that the last quantity is bounded by $C\| g \| _{H^ {-1}}$.  This yields a bound $\| \text{1st term 2nd line \eqref{el1}}\| _{L^1_tL^2_x}\lesssim C_0^L \epsilon ^L$.

  \bigskip \textit{{Proof of {Lemma} \ref{lem:conditional4.2} }}
  We rewrite \eqref{eq:eqh} as
  \begin{equation*}
\begin{aligned}
  \im  \dot h &= [\mathcal{K}   h  + A (P_+ -P_-) ] h+ A [\widetilde{P}_c \sigma _3-P_+ +P_-] \sigma _3h +  \sum _{\mathbf{e}      \cdot(\mu-\nu)\in \sigma _e(\mathcal{L} _{\omega _1 })}
z^\mu \overline{z}^\nu       \mathbf{G}_{\mu \nu}    +  \widetilde{P}_c\mathbf{E}  .
\end{aligned}
\end{equation*}
Then we have
\begin{equation}\begin{aligned}
    h (t) &= \mathcal{U}(t,0) e^{\im t\mathcal{K}}h(0) +\int _0^t\mathcal{U}(t,s)e^{\im (t-s)\mathcal{K}}\left [ A    [\widetilde{P}_c \sigma _3-P_+ +P_-] \sigma _3h \right . \\& \left .+ \sum _{\mathbf{e}      \cdot(\mu-\nu)\in \sigma _e(\mathcal{L} _{\omega _1 })}
z^\mu \overline{z}^\nu       \mathbf{G}_{\mu \nu}    +  \widetilde{P}_c\mathbf{E} \right ] ds,
\end{aligned}
\end{equation}
where the following operator commutes with $\mathcal{K}$:
\begin{equation*}
 \mathcal{U}(t,s)= ^{\im \int _s^t A(s') ds'(P_+ -P_-) }.
\end{equation*}
Then
 \begin{equation*}
\begin{aligned}
 \|  h \| _{L^p_t W^{ 1 ,q}_x \cap L^2_t H^{ 1 ,-s}_x } & \lesssim
 \| h (0)\| _{H^1}  +\sum _{\mu \nu}\| z^\mu \overline{z}^\nu  \| _{L^2_t }\|  \mathbf{G}_{\mu \nu}\| _{L^\infty_t H^{ 1 , s}_x}\\&
 + \| A \| _{L^\infty_t } \|  h \| _{ L^2_t H^{ 1 ,-s}_x }
 + \| R_1 \| _{L^1_t H^{ 1 }_x } +\|
   R_2 \| _{L^{2
 }_t H^{ 1
,
 s}_x }    .
\end{aligned}
\end{equation*}
The terms on the second line are $O(\epsilon ^2)$ and   the r.h.s. is
bounded by the r.h.s. of \eqref{4.5}, proving {Lemma} \ref{lem:conditional4.2}.

\section*{Acknowledgments}   S.C. was partially funded  by    grants FIRB 2012 (Dinamiche Dispersive) from the Italian Government,   FRA 2013 and FRA 2015 from the University of Trieste.
M.M. was supported by the Japan Society for the Promotion of Science (JSPS) with the Grant-in-Aid for Young Scientists (B) 15K17568.

Department of Mathematics and Geosciences,  University
of Trieste, via Valerio  12/1  Trieste, 34127  Italy. {\it E-mail Address}: {\tt scuccagna@units.it}
\\

Department of Mathematics and Informatics,
Faculty of Science,
Chiba University,
Chiba 263-8522, Japan.
{\it E-mail Address}: {\tt maeda@math.s.chiba-u.ac.jp}

\end{document}